\documentclass[11pt,letterpaper,reqno]{amsart}
\usepackage[includehead,includefoot,margin=20mm]{geometry}
\usepackage{times}
\usepackage{amsfonts} %
\usepackage{amsmath} %
\usepackage{amssymb} %
\usepackage{amsthm} %
\usepackage{enumerate} %
\usepackage[all]{xy}
\usepackage{bbm} %
\usepackage{graphicx} %
\usepackage{nicefrac} %
\usepackage{color} %
\usepackage{hyperref}
\hypersetup{
    colorlinks=true,       
    linkcolor=blue,          
    citecolor=blue,        
    filecolor=blue,      
    urlcolor=blue           
}
\usepackage[nameinlink,capitalize,noabbrev]{cleveref}

\usepackage{calrsfs}

\usepackage{pstricks}
\usepackage{tikz}
\usetikzlibrary{calc,intersections} 

\usepackage{pgfplots}
\usepgfplotslibrary{patchplots}
\usetikzlibrary{patterns, positioning, arrows}
\pgfplotsset{compat=1.15}

\usepackage{graphicx}

\newtheorem{theorem}{Theorem}[section] %
\newtheorem{corollary}[theorem]{Corollary} %
\newtheorem{lemma}[theorem]{Lemma} %
\newtheorem{proposition}[theorem]{Proposition} %
\newtheorem{conjecture}[theorem]{Conjecture} %
{\theoremstyle{remark} %
  \newtheorem{remark}[theorem]{Remark}} %
{\theoremstyle{definition} %
  \newtheorem{definition}[theorem]{Definition} %
  \newtheorem{example}[theorem]{Example} %
}

\newtheorem{introthm}{Theorem}

\newcommand\Q{\mathbf{Q}}
\newcommand\C{\mathbf{C}}
\newcommand\R{\mathbf{R}}

\newcommand{\git}{\mathbin{
  \mathchoice{/\mkern-6mu/}
    {/\mkern-6mu/}
    {/\mkern-5mu/}
    {/\mkern-5mu/}}}

\newcommand{\PP}[0]{\ensuremath{\mathbf{P}}}
\newcommand{\CC}[0]{\ensuremath{\mathbf{C}}}
\newcommand{\ZZ}[0]{\ensuremath{\mathbf{Z}}}

\newcommand{\tvarphi}[0]{\ensuremath{\widetilde{\varphi}}}

\newcommand{\ttau}[0]{\ensuremath{\widetilde{\tau}}}

\newcommand\h{\operatorname{H}}

\newcommand{\res}[0]{\ensuremath{\operatorname{res}}}

\newcommand{\Fix}[0]{\ensuremath{\operatorname{Fix}}}
\newcommand{\Aut}[0]{\ensuremath{\operatorname{Aut}}}

\newcommand{\prim}[0]{\ensuremath{\operatorname{prim}}}
\newcommand{\Lin}[0]{\ensuremath{\operatorname{Lin}}}

\newcommand{\rank}[0]{\ensuremath{\operatorname{rank}}}
\newcommand{\drank}[0]{\ensuremath{\operatorname{diff.rank}}}
\newcommand{\spar}[0]{\ensuremath{\operatorname{Spar}}}
\newcommand{\vars}[0]{\ensuremath{\operatorname{Vars}}}
\newcommand{\PGL}[0]{\ensuremath{\operatorname{PGL}}}
\newcommand{\GL}[0]{\ensuremath{\operatorname{GL}}}
\newcommand{\GP}[0]{\ensuremath{\operatorname{GP}}}
\newcommand{\PGP}[0]{\ensuremath{\operatorname{PGP}}}
\newcommand{\GT}[0]{\ensuremath{\operatorname{GT}}}
\newcommand{\PGT}[0]{\ensuremath{\operatorname{PGT}}}

\newcommand{\Gal}[0]{\ensuremath{\operatorname{Gal}}}

\begin{document}

\title[On a Torelli Principle for Klein hypersurfaces]{On a Torelli Principle for automorphisms of Klein hypersurfaces}

\author[V. Gonz\'alez-Aguilera]{V\'\i ctor Gonz\'alez-Aguilera}
\address{Departamento de Matem\'atica, Universidad T\'ecnica
  Fe\-de\-ri\-co San\-ta Ma\-r\'\i a, Valpara\'\i
  so, Chile}  \email{victor.gonzalez@usm.cl}

\author[A. Liendo]{Alvaro Liendo} %
\address{Instituto de Matem\'atica y F\'isica, Universidad de Talca,
  Casilla 721, Talca, Chile} %
\email{aliendo@utalca.cl}

\author[P. Montero]{Pedro Montero}
\address{Departamento de Matem\'atica, Universidad T\'ecnica
  Fe\-de\-ri\-co San\-ta Ma\-r\'\i a, Valpara\'\i
  so, Chile}  \email{pedro.montero@usm.cl}

\author[{R. Villaflor Loyola}]{Roberto Villaflor Loyola}
\address{Departamento de Matem\'atica, Universidad T\'ecnica
  Fe\-de\-ri\-co San\-ta Ma\-r\'\i a, Valpara\'\i
  so, Chile}  
\email{roberto.villaflor@usm.cl}

\date{\today}

\thanks{{\it 2020 Mathematics Subject
    Classification}: 14C30, 14C34, 14J50, 14J70.\\
  \mbox{\hspace{11pt}}{\it Key words}: Automorphism groups of smooth hypersurfaces, automorphisms of Hodge structures.\\
  \mbox{\hspace{11pt}} The second and third authors were partially supported by Fondecyt Projects 1200502 and 1231214. The fourth author was supported by the Fondecyt ANID postdoctoral grant 3210020.}

\begin{abstract}
Using a refinement of the differential method introduced by Oguiso and Yu, we provide effective conditions under which the automorphisms of a smooth degree $d$ hypersurface of $\mathbf{P}^{n+1}$ are given by generalized triangular matrices. Applying this criterion we compute all the remaining automorphism groups of Klein hypersurfaces of dimension $n\geq 1$ and degree $d\geq 3$ with $(n,d)\neq (2,4)$. We introduce the concept of extremal polarized Hodge structures, which are structures that admit an automorphism of large prime order. Using this notion, we compute the automorphism group of the polarized Hodge structure of certain Klein hypersurfaces that we call of Wagstaff type, which are characterized by the existence of an automorphism of large prime order. For cubic hypersurfaces and some other values of $(n,d)$, we show that both groups coincide (up to involution) as predicted by the Torelli Principle.
\end{abstract}

\maketitle

\section*{Introduction}

Let $\mathcal{H}_d^n$ be the moduli space of smooth degree $d$ hypersurfaces in $\mathbf{P}^{n+1}$. Classically, this moduli space is constructed as the GIT quotient $\mathcal{H}_d^n=\mathcal{U}\git \operatorname{PGL}_{n+2}(\mathbf{C})$ of the open subset $\mathcal{U}$, corresponding to the complement of the discriminant divisor $\Delta\subseteq \mathbf{P}\operatorname{H}^0(\mathbf{P}^{n+1},\mathcal{O}_{\mathbf{P}^{n+1}}(d))$, by the natural action of the reductive group $\operatorname{PGL}_{n+2}(\mathbf{C})$. It is well-known after \cite{MM64,Cha78} that the automorphism group of elements in $\mathcal{H}_d^n$ is a finite subgroup of $\operatorname{PGL}_{n+2}(\mathbf{C})$ as long as $n\geq 1$, $d\geq 3$ and $(n,d)\notin \{(1,3),(2,4)\}$. The general elements in $\mathcal{H}_d^n$ have trivial automorphism group and thus they correspond to smooth points in $\mathcal{H}_d^n$. On the other hand, a finite group $G$ induces a (possibly empty) sublocus $\mathcal{H}_d^n(G)$ corresponding to smooth hypersurfaces $X\subseteq \mathbf{P}^{n+1}$ with $G\subseteq \operatorname{Aut}(X)$. In the case $G\subseteq \operatorname{PGL}_{n+2}(\mathbf{C})$ is not generated by pseudo-reflections we have $\mathcal{H}_d^n(G) \subseteq \operatorname{Sing}(\mathcal{H}_d^n)$, and thus in order to study the components of $\operatorname{Sing}(\mathcal{H}_d^n)$ it is natural to analyze which finite groups $G$ 
 are \emph{admissible} (i.e., for which $\mathcal{H}_d^n(G)\neq \varnothing$). Since the correspondence $G\mapsto \mathcal{H}_d^n(G)$ is inclusion reversing, one may expect that the minimal admissible groups give information about the irreducible components of the singular locus of the moduli space $\mathcal{H}_d^n$. In the hierarchy of finite groups the simplest ones are cyclic groups $G\simeq \mathbf{Z}\slash m\mathbf{Z}$ and their admissibility in $\mathcal{H}_d^n$ was studied in \cite{GL11,gl13,Zhe22}. Moreover in \cite{gl13} the first two authors proved that the only smooth projective hypersurface in $\mathcal{H}_d^n$ admitting an automorphism of prime order $p>(d-1)^n$ is the Klein hypersurface
\[
X=\{K:=x_0^{d-1}x_1+x_1^{d-1}x_2+\cdots+x_n^{d-1}x_{n+1}+x_{n+1}^{d-1}x_0=0\}\subseteq \mathbf{P}^{n+1},
\]
(i.e., $\mathcal{H}_d^n(\mathbf{Z}\slash p\mathbf{Z})=\{X\}$ is $0$-dimensional) and in that case $n+2$ is also prime, and $p$ corresponds to 
\[
p=\frac{(d-1)^{n+2}+1}{d}.
\]
Such prime numbers are known as generalized Wagstaff primes of base $d-1$ and it is conjectured there are infinitely many of them \cite{dubner2000primes}. For $d=-1$ they correspond to Mersenne primes, while for $d=3$ they are the classical Wagstaff primes. Because of this, we will say that the corresponding Klein hypersurface is of \emph{Wagstaff type}. In order to understand which irreducible components of $\operatorname{Sing}(\mathcal{H}_d^n)$ contain the Klein hypersurface one would like to know which other groups act faithfully on it, which leads us to analyze its full automorphism group.

The study of automorphisms of Klein hypersurfaces is a classical subject and begins with the paper by Felix Klein \cite{Klein} studying the automorphism group of the quartic Klein curve $X\subseteq \mathbf{P}^2$, where $d=4$ and $n=1$, and proving that $\operatorname{Aut}(X)\simeq \operatorname{PSL}_2(\mathbf{F}_7)$ in that case. The case of the Klein cubic curve $X\subseteq \mathbf{P}^2$ is also classical. A straightforward computation shows that $\nu:X\to X,\;(x_0:x_1:x_2)\mapsto (x_1:x_2:x_0)$ is an automorphism order 3 fixing an inflection point of the elliptic curve and thus it follows from \cite[Theorem III.10.1]{Sil86} that $\operatorname{Aut}(X)\simeq X \rtimes \mathbf{Z}\slash 6\mathbf{Z}$. In \cite{Adl78}, Adler proved that the automorphism group of the Klein cubic threefold is isomorphic to $\operatorname{PSL}_2(\mathbf{F}_{11})$. The automorphism group of the Klein cubic surface was computed by Dolgachev in \cite[Theorem 9.5.8]{Dol12}, who showed that it is $\mathfrak{S}_5$. The automorphism groups of the remaining Klein curves were computed by Harui in \cite[Proposition 3.5]{Har19} analyzing finite subgroups of $\operatorname{PGL}_3(\mathbf{C})$, while the automorphism group of the Klein quintic threefold was computed by Oguiso and Yu in \cite[Theorem~3.8]{OY19}. In order to describe the automorphism group in the latter two cases, let us introduce the following notation: Given $n\geq 1$ and $d\geq 3$ we define $$m:=\frac{(d-1)^{n+2}-(-1)^{n+2}}{d}.$$ A straightforward computation shows that the degree $d$ Klein hypersurface $X\subseteq \mathbf{P}^{n+1}$ admits an automorphism  $\sigma$ of order $m$ and an automorphism $\nu$ of order $n+2$ given by $\sigma = \operatorname{diag}(\zeta_{dm},\zeta_{dm}^{1-d},\ldots,\zeta_{dm}^{(1-d)^{n+1}})$, i.e.,
$$ \sigma\colon X\to X, \qquad  (x_0:x_1:\cdots:x_{n+1})\mapsto (\zeta_{dm} x_0:\zeta_{dm}^{1-d}x_1:\cdots:\zeta_{dm}^{(1-d)^{n+1}}x_{n+1}),
$$
and by
$$ \nu\colon X\to X, \qquad  (x_0:x_1:\cdots:x_{n+1})\mapsto (x_1:x_2:\cdots:x_{n+1}:x_0)\,,
$$
respectively. Hence, the group $$\mathcal{K}(n,d):=(\ZZ/m\ZZ)\rtimes \ZZ/(n+2)\ZZ$$ acts faithfully on $X$, and Harui and Oguiso-Yu showed that $\mathcal{K}(n,d)$ is the full automorphism group in the cases $(n,d)\in \{(1,d),(3,5)\}$, where $d\geq 5$.

In general, there are two commonly used methods to compute $\operatorname{Aut}(X)$ for an explicit hypersurface $X\subseteq \mathbf{P}^{n+1}$. The first one is based on the analysis of an induced action of this group on some lattice. For instance, the Picard lattice for cubic surfaces \cite{Dol12} or cohomology lattices for quartic surfaces (see e.g. \cite{Huy16,Kon20}) and cubic fourfolds (see e.g. \cite{LZ22}). The second one relies on a clever use of suitable differential operators in order to constrain the shapes of the matrices defining automorphisms of $X$. The latter method was used by Kontogeorgis in \cite{Kon02} to compute the automorphism group of Fermat-like varieties, by Poonen in \cite{Poo05} to exhibit explicit equations for hypersurfaces with trivial automorphism group, and by Oguiso and Yu in \cite{OY19} to  compute automophism groups of quintic threefolds. In this last work, Oguiso and Yu gave a general treatment of the method and named it as the \emph{differential method}.

The first result of this article, which relies on a refinement of the differential method, provides strong restrictions on the shape of the automorphism group of a smooth hypersurface under certain sparsity conditions on the monomials of the equation with respect to a given basis (see \cref{sparsity} and \cref{def:sparcity} for details).

\begin{introthm}[see \cref{theopseudotriang}]
\label{thm0}
Let $X=V(F)\subset \PP(V)$ be a smooth hypersurface of dimension $n\geq 1$ and degree $d\geq 3$ with $(n,d)\neq (1,3),(2,4)$. Let also $\beta$ and $\beta^*$ be dual bases of $V$ and $V^*$, respectively. If $\spar(F)>4$ and $(\beta^*,\le_F)$ is a poset then every element of $\Aut(X)$ is a generalized triangular matrix relative to the basis $\beta$ (i.e., of the form $P_1TP_2$ where $P_1$ and $P_2$ are permutation matrices and $T$ is upper triangular).
\end{introthm}

The above result is specially useful when the shape of a given equation is rather simple. For instance, we can use it to retrieve the automorphism group of most Fermat hypersurfaces (see \cref{aut-fermat}), and to compute the automorphism group of most Delsarte hypersurfaces (see \cref{aut-delsarte}) and most Klein hypersurfaces. More precisely, with the only exception of the Klein quartic surface (cf. \cref{rmkk3}), we compute the automorphism group of all remaining Klein hypersurfaces that are not covered in the literature up to our knowledge (see \cref{theoautklein1} and \cref{theoautklein2}). Thus, the automorphism groups of Klein hypersurfaces can be summarized as follows.

\begin{introthm}
\label{thmB}
    Letting $n\geq 1$ and $d\geq 3$. Then, the automorphism group of the Klein hypersurface $X$ of dimension $n$ and degree $d$ is isomorphic to 
    $$\Aut(X)=\left\{
    \begin{array}{cl}
    X\rtimes{\ZZ/6\ZZ} & \mbox{if } n=1, d=3, \\
    \operatorname{PSL}_2(\mathbf{F}_{7}) & \mbox{if } n=1, d=4, \\
    \mathcal{K}(n,d) & \mbox{if } n=1, d\geq 5, \\    
    \mathfrak{S}_5 & \mbox{if } n=2, d=3, \\
    ? \ \ (+\infty) & \mbox{if } n=2, d=4, \\
    \mathcal{K}(n,d) & \mbox{if } n=2, d\geq 5, \\
    \operatorname{PSL}_2(\mathbf{F}_{11}) & \mbox{if } n=3, d=3, \\
    \mathcal{K}(n,d) & \mbox{if } n=3, d\geq 4, \\
    \mathcal{K}(n,d) & \mbox{if } n\geq 4, d\geq 3. \\
\end{array}\right.$$
\end{introthm}

One last fruitful approach to study automorphism groups is to use Hodge theoretic methods. For example, the classical Torelli Theorem reduces the study of automorphism groups of curves to those of principally polarized abelian varieties (ppav, for short). An account of part of the history of this method is as follows: In one hand, if one wants to study the singular locus of the moduli of
genus $g$ ppav $\mathcal{A}_g$, then (similarly to the study of $\operatorname{Sing}(\mathcal{H}_d^n)$) there is a correspondence between its components and non-trivial finite groups $G$ (i.e., $G \neq \{\pm \operatorname{Id}\}$) acting on each family. As before, in the hierarchy of finite groups, one should first focus on cyclic groups and in particular in prime order automorphisms. Let $\sigma \in \operatorname{Aut}(A,\Theta)$ be an automorphism of prime order $p$ of a $g$-dimensional ppav. Then $p\leq 2g+1$ (see e.g. \cite[Prop. 13.2.5]{BL04}), and varieties for which equality holds were studied in \cite{GAMPZ05} under the name of \emph{extremal ppav}, since they correspond to certain $0$-dimensional components of $\operatorname{Sing}(\mathcal{A}_g)$. On the other hand, by the classical Torelli Theorem this kind of bound extends to smooth projective curves of genus $g$ and, more generally, to smooth projective varieties $X$ for which the middle cohomology is of Hodge level one and such that the Torelli Principle holds for the corresponding intermediate Jacobian $J(X)$. After the work of Deligne \cite{Del72} and Rapoport \cite{Rap72}, the only $n$-dimensional smooth hypersurfaces of the projective space with Hodge level one (and thus admitting a principally polarized intermediate Jacobian) are given by cubic threefolds, quartic threefolds, and cubic fivefolds. For a smooth cubic threefold $X\subseteq \mathbf{P}^4$, the corresponding Torelli Theorem is due to Clemens and Griffiths \cite{CG72}. The latter result together with techniques from \cite{GAMPZ05} on extremal ppav were used by Roulleau \cite{Rou09} to show that $\sigma \in \operatorname{Aut}(X)$ of prime order verifies $\operatorname{ord}(\sigma)\leq 11$ with equality if and only if $X$ is the Klein cubic threefold. As mentioned before, this characterization of the Klein hypersurfaces was extended to higher dimension and degree by the first two authors in \cite{gl13}. In the case of quartic threefolds and cubic fivefolds there is positive evidence for the Torelli Theorem only for \emph{generic} members in the corresponding moduli space by a celebrated result by Donagi \cite{donagi1983generic}. Also the Abel-Jacobi isomorphism between $J(X)$ and suitable Albanese varieties only holds for \emph{generic} $X$ by \cite{Col86,Let84}.

The starting point of this work was the fact that for quartic threefolds and cubic fivefolds, as in the case of cubic threefolds, the dimension $g=\dim J(X)$ of their intermediate Jacobians is such that $p:=2g+1$ is a prime number. In particular, if the quartic threefold (resp. cubic fivefold) admits an automorphism of prime order $p=61$ (resp. $p=43$) then the induced automorphism $\sigma \in \operatorname{Aut}(J(X),\Theta)$ will make $(J(X),\Theta)\in \mathcal{A}_g$ an extremal ppav in the sense of \cite{GAMPZ05} (cf. \cite[\S 8.4]{ST61}) and, as mentioned before, this situation only happens when  $X$ is isomorphic to the Klein hypersurface. In the particular case of cubic fivefolds, the corresponding prime $p=43$ is one of the few primes for which the class number of $\mathbf{Q}(\sqrt{-p})$ is 1, and thus the results from \cite{Adl81,bennama1997remarques} imply that there is a unique $21$-dimensional ppav with automorphism group given by $\operatorname{PSL}_2(\mathbf{F}_{43})$. It is natural to ask whether or not this abelian variety coincides with the intermediate Jacobian $J(X)$ of the Klein cubic fivefold, which would led to a counterexample of the Torelli Theorem for cubic fivefolds since $\operatorname{PSL}_2(\mathbf{F}_{43})$ has no faithful representations of degree $7$ (cf. \cite{GALM22}). Using the method of Bennama and Bertin \cite{bennama1997remarques} to describe automorphism groups of extremal ppav we prove that this is not the case, see \cref{example:cubic-5fold}.

More generally, the results in \cite{gl13} imply that the intermediate Jacobian $J(X)$ of the degree $d$ Klein hypersurface $X\subseteq \mathbf{P}^{n+1}$ is an extremal ppav for $(n,d)\in \{(3,3),(5,3),(3,4)\}$. However, as mentioned before, all the remaining Klein hypersurfaces of Wagstaff type do not admit a principally polarized intermediate Jacobian since their polarized Hodge structure is not of level one. This leads us to ask ourselves if there is some natural notion of extremal polarized Hodge structure verified by Klein hypersurfaces of Wagstaff type, useful for computing the group of automorphisms of polarized Hodge structures. In the second part of this article we develop the concept of extremal polarized Hodge structure, see \cref{defextrpolHS}, and study its properties. Extending the results of Benama and Bertin we are able to use this notion to compute the automorphism group of the polarized Hodge structure of some Klein hypersurfaces of Wagstaff type. Our main result can be summarized as follows.

\begin{introthm}[see \cref{thm:punctual-torelli}]
\label{thmA}
The middle primitive cohomology group of the Klein hypersurface of Wagstaff type $X$ of dimension $n\geq 3$ and degree $d\geq 3$ is an extremal polarized Hodge structure. Moreover, 
$$
\Aut(\operatorname{H}^n(X,\mathbf{Z})_{\prim})/\{\pm 1\}\simeq \Aut(X)
$$ 
in the following cases:
\begin{itemize}
    \item[(a)] $d\mid n+3$,
    \item[(b)] $d=3$ and $n\ge 5$.
\end{itemize}
\end{introthm}

The above result gives evidence supporting the so called Strong Torelli Principle for hypersurfaces. Recall that the classical Torelli Theorem \cite{And58, arbarello2013geometry} states that for any pair of  curves $X$ and $X'$ of genus $g$, every isomorphism between their polarized Jacobian varieties $J(X)\simeq J(X')$ is induced, possibly up to an involution, by a unique isomorphism of curves $X\simeq X'$. Since the Jacobian of a curve $X$ is determined by its Hodge structure $J(X)=(\h^{1,0}(X))^*/\h_1(X,\ZZ)$, an isomorphism between polarized Jacobian varieties $J(X)\simeq J(X')$ is represented by an isomorphism between the corresponding polarized Hodge structures. 

The Torelli Theorem motivated the following conjecture commonly called the Strong Torelli Principle: Let $X$ and $X'$ be two smooth hypersurfaces of degree $d$ and dimension $n$ on $\PP^{n+1}$. Then every isomorphism of polarized Hodge structures $ \varphi\colon \h^n(X,\ZZ)\to \h^n(X',\ZZ)$ preserving polarizations is induced, possibly up to an involution, by a unique isomorphism $\psi\colon X\to X'$. The weaker statement where  we only ask for the existence of the isomorphisms $\varphi$ and $\psi$ without requiring the former being induced by the latter is called the Global Torelli Principle. 

The Strong Torelli Principle holds for plane curves by the classical Torelli Theorem, for quartic surfaces \cite{pyatetskii1971torelli,burns1975torelli,friedman1984new}, for cubic threefolds \cite{CG72,beauville1982singularites} and  for cubic fourfolds \cite{voisin1986theoreme,looijenga2009period,Cha12,HR19}. Furthermore, the verification of the Rational Global Torelli Principle for generic hypersurfaces of dimension $n$ and degree $d$ in $\PP^{n+1}$ has been recently completed by Voisin \cite{voisin2022schiffer} for all values of $(n,d)$ with the exception of a finite number of cases (complementing previous works of Donagi \cite{donagi1983generic} and Cox--Green \cite{cox1990polynomial}). It is worth noting that in general one can only expect the Rational Torelli Principle to hold for generic hypersurfaces (see \cite[{Remark~0.3}]{voisin2022schiffer}). For instance, in the case of Hodge structures of level one there are non-isomorphic hypersurfaces with isogenous intermediate Jacobians (e.g. for curves). And for Hodge structures of higher level, whenever the rational Hodge structure splits (for instance in the Noether-Lefschetz locus) we get several non-trivial involutions which in general do not come from actual involutions of the hypersurface (e.g. for K3 surfaces).

If we take $X=X'$ in the Strong Torelli Principle, we obtain what we call the Punctual Torelli Principle: Let $X$ be a smooth hypersurfaces of degree $d$ and dimension $n$ on $\PP^{n+1}$. Then every automorphism of polarized Hodge structures $\h^n(X,\ZZ)$ preserving polarizations is induced, possibly up to an involution,  by a unique automorphism of $X$. In view of the generic Global Torelli Principle and the fact that a generic hypersurface has trivial automorphism group \cite{MM64}, it is natural to ask whether the Punctual Torelli Principle holds for smooth hypersurfaces with large automorphism group. Recalling that Klein hypersurfaces of Wagstaff type correspond precisely to those hypersurfaces admitting an automorphism of prime order $p>(d-1)^n$, \cref{thmA} provides a partial positive answer to this question. In other words we can restate \cref{thmA} by saying that the Punctual Torelli Principle holds for Klein hypersurfaces of Wagstaff type with $d\mid n+3$, or $d=3$ and $n\geq 5$. Beside the cases we were able to prove, we expect \cref{thmA} to hold in all the remaining cases. 

The article is organized as follows: In \S\ref{sec1} we recall the differential method, introduced by Oguiso and Yu. In \S\ref{sec2} we prove \cref{thm0}, which is a refinement of the differential method that provides effective conditions under which the automorphisms of an hypersurface are given by generalized triangular matrices. As an application, we compute in \S\ref{sec3} the automorphism groups of most Fermat, Delsarte and Klein hypersurfaces, proving in this way \cref{thmB}. In \S\ref{sec4} we introduce the notion of extremal polarized Hodge structure and we study their automorphism group in \cref{theoautextHS}. Finally, in \S\ref{sec5} we prove that the Hodge structure of Klein hypersurfaces of Wagstaff type is extremal and we compute their automorphism groups in the cases described in \cref{thmA}. We remark that part of the proof of \cref{thmA} relies on an elementary number theory question (see \cref{conj}), that we check for cubic hypersurfaces in the \hyperref[appendix]{Appendix}.

\subsection*{Acknowledgements} This collaboration started during the conference AGREGA held at Pontificia Universidad Católica de Chile in January 2022. We would like to thank the institution for the support and hospitality. The authors would like to warmly thank Michela Artebani for showing us \cite{Shi86,SI77} and to Fabián Levicán for writing the Python script used to produce \cref{table}. Finally, we would like to express our deep gratitude to the anonymous referee for the thorough review of the article and for bringing to our attention the need to consider integral Hodge structures in $\S \ref{sec4}$.

\section{The differential method}
\label{sec1}

In this section we introduce the differential method (see \cref{diff-method}) first used in \cite{Kon02,Poo05} and \cite{OY19}. In \S\ref{sec2} we will apply this method to determine the automorphism group of most hypersurfaces of Fermat, Klein and Delsarte type.

Fix a vector space $V$ over $\CC$ of dimension $n+2$, with $n\geq 1$. Let $\GL(V)$, and $\PGL(V)$ be the general linear group and the projective linear group, respectively. We denote by $\pi\colon \GL(V)\rightarrow \PGL(V)$ the canonical projection. For an automorphism $\tvarphi\colon V\rightarrow V$ in $\GL(V)$ we denote its image $\pi(\tvarphi)$ by $\varphi$. On the other hand, given an automorphism $\varphi\in \PGL(V)$ we choose a preimage by $\pi$ and we denote it by $\tvarphi\in \GL(V)$. The automorphism $\tvarphi$ induces an automorphism $\tvarphi^*\colon V^*\rightarrow V^*$ on the dual space $V^*$ given by $\tvarphi^*(L)=L\circ \tvarphi$. Consequently, it also induces a graded automorphism $\tvarphi^*\colon S(V^*)\rightarrow S(V^*)$ on the symmetric algebra $S(V^*)$ of the vector space $V^*$ given also by $\tvarphi^*(F)= F\circ\tvarphi$. This automorphism restricts to an automorphism $\tvarphi^*\colon S^d(V^*)\rightarrow S^d(V^*)$ of forms of degree $d$.

An automorphism of an algebraic variety $X$ is a regular map $X\rightarrow X$ having a regular inverse map. The group of all automorphisms of $X$ is denoted by $\Aut(X)$. Let $\PP(V)$ be the complex projective space of dimension $n+1$ of the vector space $V$. The automorphism group of $\PP(V)$ is the projective linear group $\PGL(V)$. Let now $X$ be a hypersurface of $\PP(V)$ given as the zero set of a homogeneous form $F\in S^d(V^*)$ of degree $d$, the group of linear automorphisms, denoted by $\Lin(X)$, is the subgroup of $\Aut(X)$ of automorphisms that extend to an automorphism of the ambient space $\PP(V)$, i.e., 
\begin{align*}
    \Lin(X)&=\{\varphi\in \PGL(V)\mid \varphi(X)=X\}\\
    &=\{\varphi\in
\PGL(V)\mid \tvarphi(F)=\lambda F\mbox{ for some } \lambda\in \CC^*\}\,.
\end{align*}
In general, the group $\Lin(X)$ is a proper subgroup of $\Aut(X)$, nevertheless for smooth hypersurfaces, we have the following classical theorem \cite{MM64}, see also \cite[\S 6]{Ko19} and the references therein.

\begin{theorem} \label{matsumura-monsk} %
  Let $X$ and $Y$ be smooth hypersurfaces of dimension $n\geq 1$ and degree $d\geq 3$ in the complex projective space $\PP(V)$ and let $\tau\colon X\rightarrow Y$ be an isomorphism. If $(n,d)\neq  (1,3), (2,4)$, then $\tau$ is the restriction of a linear automorphism  $\PP(V)\rightarrow \PP(V)$ in $\PGL(V)$. In particular, every automorphism of $X$ is linear.  Moreover, $\Aut(X)$ is a finite group.
\end{theorem}

The so called differential method has been used to constrain, in some particular cases via ad-hoc constructions, the shape of the matrices in $\PGL(V^*)$ defining an automorphism of a hypersurface $X$ in $\PP(V)$ given by a form having, in a certain bases, few monomials, but has never received an abstract general treatment. We intend to do that in this and the following section.

Let us fix some notations in order to do computations in
coordinates. Let $\beta=\{e_0,\ldots,e_{n+1}\}$ be a basis of
$V$ and let $\beta^*=\{x_0,\ldots,x_{n+1}\}$ be the corresponding dual basis of $V^*$. This choice also endows $V^*$ with a dot product for which $\beta^*$ is an orthonormal basis, i.e., $x_i\cdot x_j=1$ if $i=j$ and $x_i\cdot x_j=0$ if $i\neq j$. Furthermore, the choice of a basis also induces a canonical isomorphism of the symmetric algebra $S(V^*)$ and the polynomial ring
$\CC[x_0,\ldots,x_{n+1}]$. Under this isomorphism, a form
$F \in S(V^*)$ of degree $d$ corresponds to a homogeneous polynomial of total degree $d$. Moreover, after this choice of basis, an element $\tvarphi$ in $\GL(V)$ is given by a matrix $(\tvarphi_{ij})$.

\begin{definition}
\label{diff-rank}
Let $D$ be the directional derivative operator 
$$D\colon V^*\times S(V^*)\to S(V^*),\qquad  (x,F)\mapsto
\frac{\partial F}{\partial x}=\nabla(F)\cdot x\,.$$
For a fixed homogeneous form $F\in S(V^*)$  define the
\emph{specialization of $D$} given by 
$$D_F\colon V^*\to S(V^*),\qquad  x\mapsto
\frac{\partial F}{\partial x}=\nabla(F)\cdot x\,.$$ 
Recall that the rank of a linear operator is the dimension of its image. We define the \emph{differential rank} of a form $F\in S(V^*)$ as $\drank(F):=\rank D_F$.
\end{definition}

The key idea of the differential method is that the differential rank does not change if we replace $F$ by
$\tvarphi^*(F)$, where $\tvarphi\in \GL(V)$.

\begin{proposition}[{\cite[Theorem~3.5]{OY19}}] \label{diff-method-oguiso}
  Let $F,G\in S(V^*)$ and let $\tvarphi\in \GL(V)$ be such that $\tvarphi^*(G)=\lambda F$ for some $\lambda \in \CC^*$. Then $\drank(F)=\drank(G)$.
 \end{proposition}

\begin{proof}
 The following diagram 
 \begin{align*}
\xymatrix{  V^* \ar[rr]^{D_F} \ar[d]_{\tvarphi^*} & & S(V^*) \ar[d]^{\lambda(\tvarphi^*)^{-1}}  \\  V^* \ar[rr]^-{D_G} & & S(V^*) } 
\end{align*}
 is commutative. Indeed, a straightforward application of the chain rule shows that 
  \begin{align}
  \begin{split} \label{differential-method-equation}
    \lambda\frac{\partial F}{\partial x}&= \nabla(\lambda F)\cdot x =
    \nabla(G\circ \tvarphi^*)\cdot x =
    ((\nabla G)\circ
    \tvarphi^*)\cdot \tvarphi^*(x)\\
    &=(\nabla(G)\cdot \tvarphi^*(x))\circ
    \tvarphi^*=\frac{\partial G}{\partial(\tvarphi^* x)}\circ\tvarphi^*=\tvarphi^*\left(\frac{\partial G}{\partial(\tvarphi^* x)}\right)\,.
  \end{split}\
  \end{align}
  Hence, the ranks of $D_F$ and $D_G$ agree since both vertical arrows are isomorphisms.
\end{proof}

As a corollary we obtain the following lemma that we will use to constrain the shape of the automorphism groups of certain hypersurfaces including Klein hypersurfaces.

\begin{corollary}
   \label{diff-method}
  Let $X$ be a hypersurface in $\PP(V)$ given by a homogeneous form
  $F\in S(V^*)$ and let $\varphi\in PGL(V)$ be a linear automorphism
  of $X$. Then for every $x\in V^*$ we have
  $$\drank\left(\frac{\partial F}{\partial
      x}\right)=
  \drank\left(\frac{\partial F}{\partial(\tvarphi^* x)}\right).$$
\end{corollary}

\begin{proof}
  Since $\varphi$ is an automorphism of $X$, we have that
  $\tvarphi^*(F)=F\circ \tvarphi =\lambda F$ for some
  $\lambda \in \CC^*$. Now the statement follows from \cref{diff-method-oguiso} since \eqref{differential-method-equation} with $G=F$ yields 
  \begin{align*}
    \lambda\frac{\partial F}{\partial x}=\tvarphi^*\left(\frac{\partial F}{\partial(\tvarphi^* x)}\right).
  \end{align*}
\end{proof}

\begin{remark}
For a generic homogeneous form $F\in S(V^*)$ we have $\ker D_F=\{0\}$ and so the differential rank of $F$ is
$\dim(V)=n+2$. In particular, if $X$ is smooth of degree greater than one then the corresponding form $F$ has differential rank $\dim(V)$. Nevertheless, we will apply \cref{diff-method} in the cases where the differential rank of $\frac{\partial F}{\partial x}$ for a particular choice of $x\in V^*$ drops, this puts constrains on the shape of the matrices that appear as automorphisms of $X$ since the differential rank of $\frac{\partial F}{\partial (\tvarphi^* x)}$ must also drop.    
\end{remark}

\section{Generalized permutation automorphisms}
\label{sec2}

In this section we apply the differential method to restrict the automorphism group of certain hypersurfaces which we call ``simple'' (see \cref{simple-hyp} for details) and which are inspired by the main result of \cite{Zhe22}. In many cases, these restrictions are enough to determine the full automorphism group.

\begin{definition} \
\begin{enumerate}
\item[$(i)$]   A matrix  is a \emph{generalized permutation matrix} if there is at most one coefficient different from 0 in each row and each column. We denote by $\GP(V,\beta)$  the subgroup of all $\tvarphi\in \GL(V)$ such that $\tvarphi$ corresponds to a generalized permutation matrix with respect to the basis $\beta$. We also denote by $\PGP(V,\beta)$ the image of $\GP(V,\beta)$ in $\PGL(V)$. Note that every generalized permutation matrix is the product of a diagonal matrix and a permutation matrix. 

\item[$(ii)$]  A matrix is a \emph{generalized triangular matrix} if it is of the form $P_1TP_2$ where $P_1$ and $P_2$ are permutation matrices, and $T$ is upper triangular. We denote by $\GT(V,\beta)$ the subset of all $\tvarphi\in\GL(V)$ such that $\tvarphi$ corresponds to a generalized triangular matrix with respect to the basis $\beta$. We also denote by $\PGT(V,\beta)$ the image of $\GT(V,\beta)$ in $\PGL(V)$. 
\end{enumerate}
\end{definition}

The groups $\GP(V,\beta)$, $\PGP(V,\beta)$, and the sets $\GT(V,\beta)$ and $\PGT(V,\beta)$ are indeed dependent on the choice of a basis as the notation suggests.

In order to show that the automorphism group of most simple hypersurfaces consists of generalized permutations, we need a notion of sparsity of the set of monomials appearing in a homogeneous form in a given basis.

\begin{definition}
\label{sparsity}
Let $F\in S^d(V^*)$, and $\beta$ and $\beta^*$ be dual bases of $V$ and $V^*$, respectively.
\begin{enumerate}
    \item[$(i)$] We define the distance between two monomials $\prod x_i^{a_i}$ and $\prod x_i^{b_i}$ as the sum $\sum_{i}|a_i-b_i|$. 
    \item[$(ii)$] We define the \emph{sparsity} of $F$ with respect to $\beta$, denoted by $\spar(F)$, as the minimum of the distance between the monomials of $F$ after identifying $S(V^*)\simeq\C[x_0,\ldots,x_{n+1}]$ via the basis $\beta$.    
    \item[$(iii)$] We define the variables of $F$, denoted by $\vars(F)$, as the set of variables appearing in $F$, i.e., $\vars(F)$ is the smallest subset of $\beta^*$ such that $F$ is contained in $S^d(W)$ with $W$ the span of $\vars(F)$.
\end{enumerate}

\end{definition}

Note that the sparsity of a homogeneous form is always an even non-negative integer. When the sparsity of a form is greater than 2, we can compute its differential rank directly as we show in the following lemma

\begin{lemma} \label{lem:vars}
Let $F\in S(V^*)$ be homogeneous of degree $d\geq 2$. If $\spar(F)>2$, then $\drank(F)$ equals the cardinality of $\vars(F)$.
\end{lemma}

\begin{proof}
Let where $I=\beta^*\setminus \vars(F)$. Now, note that $\ker D_F$ equals the span of $I$. This follows from the fact that $\frac{\partial F}{\partial x_j}$ and $\frac{\partial F}{\partial x_j}$ for $i,j=0,\ldots,n+1$ with $i\neq j$ never share a common monomial due to $\spar(F)>2$.  
\end{proof}

We now define a relation on the dual basis $\beta^*$ of $V^*$.

\begin{definition}\label{def:sparcity}
Let $F\in S(V^*)$ be homogeneous of degree $d\geq 2$, and $\beta$ and $\beta^*$ be dual bases of $V$ and $V^*$, respectively. We endow the set $\beta^*=\{x_0,\ldots,x_{n+1}\}$ with the relation $\le_{F}$ given by
$$
x_i\le_F x_j \quad \Longleftrightarrow \quad \vars\left(\frac{\partial F}{\partial x_i}\right)\subseteq \vars\left(\frac{\partial F}{\partial x_j}\right)
$$
\end{definition}

\begin{remark}
The set $\beta^*$ together with the relation $\le_F$ may not be a partial order. We call a partially ordered set a poset for brevity. Indeed, let $\beta^*=\{x_0,x_1,x_2,x_3\}$ and consider first 
$$F=x_0^{d-1}x_1+x_1^{d-1}x_0+x_2^d+x_3^d.$$ 
Then $(\beta^*,\le_F)$ is not a poset since $x_1\le_{(F,\beta)}x_2$ and $x_2\le_{(F,\beta)}x_1$. On the other hand, if we consider instead
$$F=x_0^{d-1}x_1+x_1^{d-1}x_2+x_2^{d-1}x_3+x_3^{d-1}x_0,$$ we have that $(\beta^*,\le_F)$ is a poset since $x_i$ and $x_j$ are not comparable whenever $i\neq j$. We say that such a poset, where $x_i\le_F x_j$ implies $i=j$, is trivial.
\end{remark}

We now come to the main theorem of this section which provides strong restrictions on the shape of the automorphism group of a smooth hypersurface under certain sparsity conditions.

\begin{theorem}
\label{theopseudotriang}
Let $X=V(F)\subset \PP(V)$ be a smooth hypersurface of dimension $n\geq 1$ and degree $d\geq 3$ with $(n,d)\neq (1,3),(2,4)$. Let also $\beta$ and $\beta^*$ be dual bases of $V$ and $V^*$, respectively. If $\spar(F)>4$ and $(\beta^*,\le_F)$ is a poset then $\Aut(X)\subseteq\PGT(V,\beta)$.
\end{theorem}

\begin{proof}
    Let $\varphi\in \Aut(X)$. By \cref{matsumura-monsk}, $\varphi\in \PGL(V)$ so we can pick a representative $\tvarphi\in \GL(V)$ such that $\tvarphi^*(F)=\lambda F$. We represent $\tvarphi$ in the basis $\beta$ by a matrix $(\tvarphi_{ji})$ so that $\tvarphi^*\colon V^*\to V^*$ is given in the basis $\beta^*$ by the matrix $(\tvarphi_{ij})$. With this notation, we have 
    $$\frac{\partial F}{\partial (\tvarphi^*x_i)}=\sum_j \tvarphi_{ij} \frac{\partial F}{\partial x_j}\,.$$
    Remark  that since $\spar(F)>4$ it follows that $\spar\left(\frac{\partial F}{\partial (\tvarphi^*x_i)}\right)>2$. Hence, by \cref{lem:vars} we have
    \begin{align} \label{drank-vars}
    \drank\left(\frac{\partial F}{\partial (\tvarphi^*x_i)}\right)=\#\vars\left(\frac{\partial F}{\partial (\tvarphi^*x_i)}\right)=\#\bigcup_{\{j\mid\tvarphi_{ij}\neq 0\}} \vars\left(\frac{\partial F}{\partial x_j}\right)
    \end{align}

    By \cref{diff-method}, we have $\drank\left(\frac{\partial F}{\partial x_i}\right)=\drank\left(\frac{\partial F}{\partial (\tvarphi^*x_i)}\right)$. We claim that 
    \begin{align*}
    \drank\left(\frac{\partial F}{\partial x_i}\right)=\drank\left(\frac{\partial F}{\partial (\tvarphi^*x_i)}\right) \Longleftrightarrow \left\{\begin{array}{l}
        \tvarphi^*x_i=\tvarphi_{ij}\cdot x_{j}+\tvarphi_{ij_1}\cdot x_{j_1}+\cdots+\tvarphi_{ij_k}\cdot x_{j_k}, \\[1ex]
        \mbox{with } \drank\left(\frac{\partial F}{\partial x_i}\right)=\drank\left(\frac{\partial F}{\partial x_{j}}\right) \\[1.5ex]
        \mbox{and } x_{j_\ell}\le_F x_{j} \mbox{ for all } \ell\in \{1,\ldots,k\}. 
        \end{array}\right.
    \end{align*}
    
    Indeed, the converse implication is trivial. For short, along the proof of the claim we will write $F_{x_i}$ to denote $\frac{\partial F}{\partial x_i}$.  
    To prove the direct implication, take first a variable $x_i$ such that $F_{x_i}$ has minimal possible differential rank. Then the only possibility is $\tvarphi^*x_i=\tvarphi_{ij}x_{j}$ with $\drank(F_{x_i})=\drank(F_{x_j})$ since otherwise the differential rank would increase. Note that since $\tvarphi$ is an invertible matrix, the assignment $x_i\mapsto x_{j}$ gives a permutation of the variables $x_i$, for which $\drank (F_{x_i})$ attains the minimum.

    Proceed now by induction. Assume the claim holds for all variables with differential rank less than $N$ and that the assignment $x_i\mapsto x_{j}$ is a permutation of the set of variables $x_i$ such that $\drank(F_{x_i})< N$. 
    
    Let now $x_i$ be a variable with $\drank(F_{x_i})=N$. The image $\tvarphi^*x_i=\tvarphi_{ij}\cdot x_{j}+\tvarphi_{ij_1}\cdot x_{j_1}+\cdots+\tvarphi_{ij_k}\cdot x_{j_k}$ cannot be composed only of variables $x_j, x_{j_m}$ such that $F_{x_j}$ and $F_{x_{j_m}}$ have smaller differential rank than $F_{x_i}$ since otherwise $\tvarphi$ is not an invertible matrix. Hence, we have  $\drank(F_{x_i})=\drank(F_{x_j})$. By \eqref{drank-vars}, we have $x_{j_\ell}\le_F x_{j}$ for all $\ell\in \{1,\ldots,k\}$. Finally, since $\tvarphi$ is an invertible matrix, the assignment $x_i\mapsto x_{j}$ provides a permutation of the set of variables of differential rank equal to $N$. This proves the claim.

    To conclude the proof, let $P_1$ be the permutation matrix that is the inverse of the assignments $x_i\mapsto x_{j}$ and let $P_2$ be any permutation matrix reordering the variables so that the differential rank of $x_i$ is smaller or equal than that of $x_j$ whenever $i<j$. Now, $P_1\tvarphi P_2$ is an upper triangular matrix, proving that $\tvarphi\in \GT(V,\beta)$ and so $\varphi\in\PGT(V,\beta)$.
\end{proof}

\begin{remark} \label{rmk:poset}
The proof of \cref{theopseudotriang} provides precise constraints on the shape of a matrix $\tvarphi\in \GL(V)$ representing an automorphism of $X$. More precisely, in the triangular matrix $P_1\tvarphi P_2$, at each row $j$, the non-zero elements can only appear at the columns indexed by lower terms with respect to the poset $(\beta^*,\le_F)$. In particular, when the poset is trivial the matrix $P_1\tvarphi P_2$ is diagonal and so $\tvarphi$ is always a generalized permutation in the basis $\beta^*$. We state this fact in the following corollary.
\end{remark}

\begin{corollary}
\label{corpseudoperm}
Let $X=V(F)\subset \PP(V)$ be a smooth hypersurface of dimension $n\geq 1$ and degree $d\geq 3$ with $(n,d)\neq (1,3),(2,4)$. Let also $\beta$ and $\beta^*$ be dual bases of $V$ and $V^*$, respectively. If $\spar(F)>4$ and $(\beta^*,\le_F)$ is the trivial poset then 
$\Aut(X)\subseteq\PGP(V,\beta)$.
\end{corollary}

As an application of the previous result we will show that for simple hypersurfaces, all automorphisms are generalized permutations. Our definition of simple hypersurfaces is inspired on the main result of \cite{Zhe22}, which claims that every abelian group acting on any smooth hypersurface can be realized as a subgroup of the automorphism groups of these simple hypersurfaces.

\begin{definition}
\label{simple-hyp}
We say that a smooth hypersurface $X=V(F)$ given by a homogeneous form $F\in S^d(V^*)$ of degree $d$ is simple if, after possibly renaming the elements of the basis $\beta^*$ of $V^*$, the form $F$ is given by
\begin{equation}
\label{simpleeq}    
F=K_{i_1}+K_{i_2}+\cdots+K_{i_\ell}+T_{j_1}+\cdots+T_{j_m}\,,
\end{equation}
where $
K_i=x_1^{d-1}x_2+x_2^{d-1}x_3+\cdots+x_i^{d-1}x_1
$
is the Klein form of degree $d$ in $i$ variables,
$
T_j=y_1^{d-1}y_2+\cdots+y_{j-1}^{d-1}y_j+y_j^d
$
is the Delsarte form of degree $d$ in $j$ variables, and all such Klein and Delsarte forms in \eqref{simpleeq} have linearly independent variables. In consequence $i_1+\cdots+i_\ell+j_1+\cdots+j_m=\dim(V)=n+2$.
\end{definition}

\begin{corollary}
\label{corsemiperm}

Let $X=V(F)\subset \PP(V)$ be a simple hypersurface of dimension $n\geq 1$ and degree $d\geq 3$ with $(n,d)\neq (1,3),(2,4)$. Let also $\beta$ and $\beta^*$ be dual bases of $V$ and $V^*$, respectively. If $\spar(F)>4$ then 
$\Aut(X)\subseteq\PGP(V,\beta)$.
\end{corollary}

\begin{proof}
All hypothesis of \cref{theopseudotriang} are easily verified. Moreover, the only non-trivial relation in the poset $(\beta^*,\le_{F})$ is  $y_1\le_F y_2$ for each Delsarte form $T_j$ in $F$. Let now $\tvarphi\in \GL(V)$ represent an automorphism $\varphi\colon X\rightarrow X$. Letting $P_1$ and $P_2$ be as in the proof of \cref{theopseudotriang}, we have that $P_1\tvarphi P_2$ is an upper triangular matrix. Moreover, by \cref{rmk:poset} the matrix $P_1\tvarphi P_2$ is a diagonal by blocks matrix where the block corresponding to Klein forms is diagonal and the block corresponding to Delsarte forms is given by
$$
T=\begin{pmatrix}
D & 0 \\ 0 & M    
\end{pmatrix}
$$
where $M=\begin{pmatrix}
a & b \\ 0 & d    
\end{pmatrix}$ is a $2\times 2$ upper triangular matrix and $D$ is diagonal.
Since the matrix $\tvarphi^*$ must also preserve the equation $F$, a direct computation shows that $b=0$. Indeed, if $b\neq 0$ then the monomial $y_{j-1}^{d-2}y_{j}^2$ appears with non-zero coefficient in $\tvarphi^*(T_j)$ contradicting the fact that $\varphi$ is an automorphism of $X$. This yields that $P_1\tvarphi P_2$ is diagonal which proves that $\tvarphi$ is a generalized permutation.   
\end{proof}

\begin{remark} \
\begin{enumerate}
\item[$(i)$] 
We do not know to what extent we can improve \cref{theopseudotriang}, i.e., characterize hypersurfaces where all automorphisms are generalized triangular matrices or generalized permutation matrices in some basis. We know that this is not always the case for some sporadic varieties, such as the Klein cubic threefold, but it might be the case for most hypersurfaces. 

\item[$(ii)$] 
A folklore conjecture is that, with the exception of some finite instances of $(n,d)$, where $n$ is the dimension and $d$ is the degree, the hypersurface with largest automorphism group is the Fermat hypersurface. A first step towards this conjecture would be to show that the Fermat hypersurface is the one with the largest automorphism group among hypersurfaces whose automorphism group is given by generalized permutation matrices.
\end{enumerate}
\end{remark}

\section{Automorphism groups of classical hypersurfaces}
\label{sec3}

In this section we apply \cref{corsemiperm} to compute the automorphism group of several classical hypersurfaces.

\subsection{Fermat hypersurfaces} \label{ssec:fermat}

The Fermat hypersurface $X=V(F)$ of dimension $n$ and degree $d$ is given in the basis $\beta^*$ by the form 
$$F=x_0^{d}+x_1^{d}+x_2^{d}+x_3^{d}+ \cdots +x_{n}^{d}+x_{n+1}^{d}\in S^d(V^*)\,.$$

The following theorem is well known \cite{Shi88,Kon02}. Nevertheless, we prove it as a straightforward application of \cref{corsemiperm}. We remark that the Fermat hypersurface is simple since the form $F$ is given as the sum of $n+1$ Delsarte forms of degree $d$ in $1$ variable.

\begin{proposition}
    \label{aut-fermat}
The automorphism group of the Fermat hypersurface $X=V(F)$ of dimension $n\geq 1$ and degree $d\geq 3$, with $(n,d)\neq  (1,3), (2,4)$ is isomorphic to 
$$\Aut(X)= (\ZZ/d\ZZ)^{n+1}\rtimes S_{n+2}$$
\end{proposition}
\begin{proof}
The image in $\PGL(V)$ of every permutation matrix is an automorphism of $X$ by the symmetry of the form $F$. Let now $\varphi\in \Aut(X)$ and let $\tvarphi\in \GL(V)$ be a lifting of $\varphi$ to $\GL(V)$. Since $\spar(F)=2d> 4$, by \cref{corsemiperm}, we have that $\tvarphi=(\tvarphi_{ij})$  is given in the basis $\beta^*$ by a generalized permutation matrix. Then, after multiplying on the right by a permutation matrix we can reduce ourselves to the case where $\tvarphi$ is diagonal. Letting $D$ be the image in $\PGL(V)$ of diagonal matrices that leave $F$ invariant, we obtain that $\Aut(X)= D\rtimes S_{n+2}$. Furthermore, 
$$\varphi(F)=\tvarphi^d_{00}x_0^{d}+\tvarphi^d_{11}x_1^{d}+\tvarphi^d_{22}x_2^{d}+\tvarphi^d_{33}x_3^{d}+ \cdots +\tvarphi^d_{nn}x_{n}^{d}+\tvarphi^d_{n+1,n+1}x_{n+1}^{d}\,.$$
This yields that each $\tvarphi_{ii}$ is a $d$-root of unity, $i\in\{1,\ldots,n+1\}$. Hence, we have an epimorphism $(\mathbf{Z}/d\mathbf{Z})^{n+2}\rightarrow D$ whose kernel is the diagonal $\Delta\simeq \mathbf{Z}/d\mathbf{Z}$. This yields $D\simeq (\mathbf{Z}/d\mathbf{Z})^{n+1}$. 
\end{proof}

\subsection{Delsarte hypersurfaces}

The Delsarte hypersurface $X=V(T)$ of dimension $n$ and degree $d$ is given in the basis $\beta^*$ by the form 
$$T=x_0^{d-1}x_1+x_1^{d-1}x_2+x_2^{d-1}x_3+x_3^{d-1}x_4+ \cdots +x_{n}^{d-1}x_{n+1}+x_{n+1}^{d},$$

\begin{proposition}\label{aut-delsarte}
The automorphism group of the Delsarte hypersurface $X$ of dimension $n\geq 2$ and degree $d\geq 4$, with $(n,d)\neq (2,4)$ is isomorphic to 
$$\Aut(X)= \ZZ/(d-1)^{n+1}\ZZ$$
\end{proposition}

\begin{proof}
Let $\varphi\in \Aut(X)$ and let $\tvarphi\in \GL(V)$ be a lifting of $\varphi$ to $\GL(V)$. Since $\spar(F)=2d-2>4$, by \cref{corsemiperm}, we have that $\tvarphi=(\tvarphi_{ij})$  is given in the basis $\beta^*$ by a generalized permutation matrix. By the distinguished role of the variable $x_{n+1}$, a direct inspection of the form $T$ yields that $\tvarphi$ must be a diagonal matrix.
Applying $\varphi$ to $F$ we obtain the equations
$$\tvarphi_{i,i}^{d-1}\tvarphi_{i+1,i+1}= \tvarphi_{n+1,n+1}^d\quad \forall i\in\{0,\ldots,n\}\,.$$
Up to changing $\varphi$ by a multiple, we can assume $\tvarphi_{n+1,n+1}^d=1$. It follows that $\varphi$ is determined by the value of $\tvarphi_{00}$, which in turn satisfies $\tvarphi_{00}^{d(1-d)^{n+1}}=1$. Hence, we have an epimorphism $\ZZ/d(d-1)^{n+1}\ZZ\rightarrow\Aut(X)$ whose kernel is the $d$-roots of unity, proving the proposition.
\end{proof}

\subsection{Klein hypersurfaces}

The Klein hypersurface $X=V(K)$ of dimension $n$ and degree $d$ is given in the basis $\beta^*$ by the form 
$$K=x_0^{d-1}x_1+x_1^{d-1}x_2+x_2^{d-1}x_3+x_3^{d-1}x_4+ \cdots +x_{n}^{d-1}x_{n+1}+x_{n+1}^{d-1}x_0,$$

\begin{proposition}
\label{theoautklein1}
The automorphism group of the Klein hypersurface $X$ of dimension $n\geq 2$ and degree $d\geq 4$, with $(n,d)\neq (2,4)$ is isomorphic to 
$$\Aut(X)= (\ZZ/m\ZZ)\rtimes \ZZ/(n+2)\ZZ$$
where $m=\frac{(d-1)^{n+2}-(-1)^{n+2}}{d}.$
\end{proposition}

\begin{proof}
Let $\varphi\in \Aut(X)$ and let $\tvarphi\in \GL(V)$ be a lifting of $\varphi$ to $\GL(V)$. Since $\spar(K)=2d-2>4$, by \cref{corsemiperm}, we have that $\tvarphi=(\tvarphi_{ij})$  is given in the basis $\beta^*$ by a generalized permutation matrix.
By the structure of the form $K$,
the only generalized permutation matrices inducing automorphisms of $X$ have the shape of a multiple of
$$
T=\begin{pmatrix}
0 & * & 0 &\cdots & 0 \\  
0 & 0 & * &\cdots & 0 \\  
\vdots & \vdots & \vdots & \ddots & \vdots\\
0 & 0 & 0 &\cdots & * \\
* & 0 & 0 &\cdots & 0 
\end{pmatrix}
$$
In particular, taking all coefficients to be $1$ we obtain an automorphism inducing a cyclic permutation of the variables. Such element generates a cyclic subgroup $P<\Aut(X)$ of order $n+2$. Furthermore, composing any  $\varphi$ with an element of $P$ we obtain a diagonal matrix.

The diagonal subgroup $D<\Aut(X)$ was determined in \cite[Lemma~3.1]{Zhe22}. For the reader's convenience we will reproduce the argument here. Assume that $\varphi$ is diagonal, and note that, up to changing the representative in $\GL(V)$, we can assume that $\tvarphi_{0,0}^{d-1}\tvarphi_{1,1}=1$. Replacing in Klein's equation we get that $\tvarphi_{i,i}^{d-1}\tvarphi_{i+1,i+1}=1$ for all $i=0,\ldots,n+1$. Hence
$$
\tvarphi_{0,0}=\tvarphi_{n+1,n+1}^{1-d}=\tvarphi_{n,n}^{(1-d)^2}=\cdots=\tvarphi_{1,1}^{(1-d)^{n+1}}=\tvarphi_{0,0}^{(1-d)^{n+2}}\,.
$$
Thus $\varphi$ is determined by the value of $\tvarphi_{0,0}$ which is a $dm$-root of unity. Hence, we have an epimorphism $\mathbf{Z}/dm\mathbf{Z}\rightarrow D$ whose kernel is the $d$-roots of unity. Therefore $D$ is a cyclic group of order $m$, proving the proposition.
\end{proof}

\begin{remark}
\label{remarkautdiag}
According to the above proof, the generator of the of the action $D=\ZZ/m\ZZ$ on $X$ is
$$
\sigma(x_0:x_1:\cdots:x_{n+1})=(\zeta_{dm} x_0:\zeta_{dm}^{1-d}x_1:\cdots:\zeta_{dm}^{(1-d)^{n+1}}x_{n+1}).
$$
Nevertheless, the map $r\in\ZZ/m\ZZ\mapsto r\cdot d\in (\ZZ/dm\ZZ)/\langle m\rangle$ is an isomorphism if and only if $\gcd(d,m)=1$. Hence in this case the action of a generator of $\ZZ/m\ZZ$ on the Klein variety $X$ is given by
$$
\sigma(x_0:x_1:\cdots:x_{n+1})=(\zeta_m x_0:\zeta_m^{1-d}x_1:\cdots:\zeta_m^{(1-d)^{n+1}}x_{n+1}).
$$
\end{remark}

The three propositions above settle the problem of finding the automorphism groups of Fermat, Klein and Delsarte hypersurfaces, in all degrees greater or equal than 4. In the remaining of this section, we will concentrate on cubic Klein hypersurfaces. In this case, the automorphism groups of the Klein cubic curve \cite[Theorem III.10.1]{Sil86}, the Klein cubic surface \cite[Theorem 9.5.8]{Dol12} and the Klein cubic threefold \cite{Adl78} are known. Using the differential method we can still treat the case of the Klein cubic hypersurface of dimension $n\ge 4$ with a detailed case by case analysis.

\begin{theorem}
\label{theoautklein2}
The automorphism group of the Klein cubic hypersurface $X$ of dimension $n\geq 4$ is isomorphic to 
$$\Aut(X)= (\ZZ/m\ZZ)\rtimes \ZZ/(n+2)\ZZ$$
where $m=\frac{2^{n+2}-(-1)^{n+2}}{3}.$    
\end{theorem}

\begin{proof}
In order to conclude with the same argument as in \cref{theoautklein1} it is enough to show that all automorphisms are generalized permutations. For this we have to find a finer argument than \cref{corpseudoperm} since in this case $\spar(K)=4$. 

Let $\varphi\in \Aut(X)$ and let $\tvarphi\in \GL(V)$ be a lifting of $\varphi$ to $\GL(V)$. Let also $x=\sum_{i=0}^{n+1}c_i\cdot x_i$ be a vector in $V^*$ represented in the basis $\beta^*$. We claim that 
$$
\drank \left(\frac{\partial K}{\partial x}\right)=3 \ \ \Longleftrightarrow \ \ x=c_i\cdot x_i \ \mbox{ for some }i\in \{0,\ldots, n+1\}\,.
$$
We remark that the proposition follows directly  from the claim since by \cref{diff-method} we have that $\tvarphi^*(x_i)=\tvarphi_{ij}x_j$ and so $\Aut(X)<\PGP(V,\beta)$ as desired.

To simplify the notation, in the proof of the claim we  use cyclic convention for sub-indices: $n+2=0$, $n+3=1$, $-1=n$, $-2=n-1$ and so on. We will now compute the differential rank of $\frac{\partial K}{\partial x}$. Recall that by \cref{diff-rank} the $\drank \left(\frac{\partial K}{\partial x}\right)$  is the rank of the linear map
$$\Delta:=D_{\frac{\partial K}{\partial x}}: V^*\to S(V^*)\,.$$
Nevertheless, in this case, since $K$ is a cubic form, we have that the image of $\Delta$ is contained in $V^*$. Hence, $\Delta$ is represented by a matrix in the basis $\beta^*$ that we also denote by $\Delta$ and so 
$$\drank \left(\frac{\partial K}{\partial x}\right)=\rank(\Delta)\,.$$ 

In order to prove the claim let us separate into three cases:

\begin{enumerate}
    \item \emph{There exists some sub-index $i$ such that $c_{i-1}\cdot c_{i+1}\neq 0$:} In this case since
    $$
    \frac{d(\frac{\partial K}{\partial x})}{dx_j}=\frac{\partial }{\partial x}(x_{j-1}^2+2x_jx_{j+1})=2(c_{j-1}\cdot x_{j-1}+c_{j+1}\cdot x_j+c_j\cdot x_{j+1}),
    $$
    the $4\times 6$ minor of $\Delta$ corresponding to the partial derivatives with respect to $x_{i-2}$, $x_i$, $x_{i+1}$ and $x_{i+2}$ with columns corresponding to the linear monomials $x_{i-2}$, $x_{i-1}$, $x_i$, $x_{i+1}$, $x_{i+2}$ and $x_{i+3}$ is of the form
    $$
    M=\begin{pmatrix}
    c_{i-1} & c_{i-2} & 0 & 0 & 0 & * \\
    0 & c_{i-1} & c_{i+1} & c_i & 0 & 0 \\
    0 & 0 & c_i & c_{i+2} & c_{i+1} & 0\\
    0 & 0 & 0 & c_{i+1} & c_{i+3} & c_{i+2}
    \end{pmatrix}.
    $$
    If $c_i\neq 0$ then $\rank(M)=4$. If $c_i=0$ and $\det\begin{pmatrix}
    c_{i+2} & c_{i+1} \\ c_{i+1} & c_{i+3}    
    \end{pmatrix}\neq 0$ then we also have $\rank(M)=4$. If $c_i=0$ and $c_{i+2}c_{i+3}=c_{i+1}^2\neq 0$ then $c_{i+2}\neq 0$ and so
    $$
    \rank(M)=\rank \begin{pmatrix}
    c_{i-1} & c_{i-2} & 0 & 0 & 0 & * \\
    0 & c_{i-1} & c_{i+1} & 0 & 0 & 0 \\
    0 & 0 & 0 & c_{i+2} & c_{i+1} & 0\\
    0 & 0 & 0 & 0 & 0 & c_{i+2}
    \end{pmatrix}=4. 
    $$
    We obtain in this case $\rank(\Delta)>3$.

    \medskip
    
    \item \emph{There exists $i$ and $3\le \ell\le \frac{n}{2}+1$ such that $c_i\cdot c_{i+\ell}\neq 0$ and $c_{i+1}=c_{i+2}=\cdots=c_{i+\ell-1}=0$:} In this case the $4\times 4$ minor of $\Delta$ corresponding to the partial derivatives with respect to $x_{i-1}$, $x_{i+1}$, $x_{i+\ell-1}$ and $x_{i+\ell+1}$ with columns corresponding to the linear monomials $x_{i-1}$, $x_i$, $x_{i+\ell-1}$ and $x_{i+\ell}$ is of the form
    $$
    M=\begin{pmatrix}
    c_{i} & c_{i-1} & 0 & 0 \\
    0 & c_{i} & 0 & 0 \\
    0 & 0 & c_{i+\ell} & 0\\
    0 & 0 & 0 & c_{i+\ell}
    \end{pmatrix}.
    $$
    Thus $\rank(\Delta)>3$.

    \medskip
    
    \item \emph{There are some $c_i\cdot c_{i+1}\neq0$ and $c_j=0$ for all $j\neq i,i+1$:} In this case take the $4\times 4$ minor of $\Delta$ corresponding to the partial derivatives with respect to $x_{i-1}$, $x_i$, $x_{i+1}$ and $x_{i+2}$ with columns corresponding to the linear monomials $x_{i-1}$, $x_i$, $x_{i+1}$ and $x_{i+2}$
    $$
    M=\begin{pmatrix}
    c_{i} & 0 & 0 & 0 \\
    0 & c_{i+1} & c_i & 0 \\
    0 & c_i & 0 & c_{i+1} \\
    0 & 0 & c_{i+1} & 0 
    \end{pmatrix}. 
    $$
    Thus in this last case we also obtain $\rank(\Delta)>3$.\qedhere
\end{enumerate}
\end{proof}

\begin{remark}
A similar argument as above should work to show that automorphisms of Delsarte cubics (and more generally any simple cubic hypersurface) are generalized permutations when the dimension is big enough.
It is worthy to remark that in the proof above we used $n\ge 4$ in case (1). This is not true anymore for $n\le 3$, in fact the automorphism group of the Klein cubic threefold is $\operatorname{PSL}_2(\mathbf{F}_{11})$ \cite{Adl78} and for the Klein cubic surface it is $\mathfrak{S}_5$ (see e.g. \cite[Theorem 9.5.8]{Dol12}).   
\end{remark}

\begin{remark} Let us denote by $X=\{x_0^2x_1+\cdots+x_5^2x_0=0\}\subseteq \mathbf{P}^5$ the Klein cubic fourfold, and by $G:=\Aut(\h^4(X,\ZZ),\theta^2)$ the automorphism group of polarized primitive Hodge structure preserving $\theta^2$, where $\theta$ is the K\"ahler $(1,1)$-form of $X$. By Voisin's Torelli Theorem \cite{voisin1986theoreme} we know that $\Aut(X)\simeq G$ and by \cref{theoautklein2} $$ \Aut(X)=\langle \sigma\rangle \rtimes \langle \nu\rangle\simeq(\ZZ/21\ZZ)\rtimes \ZZ/6\ZZ  $$ where  $$ \sigma=\operatorname{diag}(\zeta_{63},\zeta_{63}^{-2},\zeta_{63}^4,\zeta_{63}^{-8},\zeta_{63}^{16},\zeta_{63}^{-32}) $$ is the diagonal automorphism of order 21, and  $$ \nu(x_0:x_1:\cdots:x_5)=(x_1:x_2:\cdots:x_5:x_0) $$ is the cyclic permutation of order 6. Recall that for cubic fourfolds we can decompose $G$ as an extension of the normal subgroup $G_s$ of symplectic automorphism by the non-symplectic ones $\overline{G}$ $$ 1\rightarrow G_s\rightarrow G\rightarrow \overline{G}\rightarrow 1. $$ Moreover, $\overline{G}$ is a cyclic group (see e.g. \cite[\S 6]{LZ22}). One can easily see that  $$ \langle \sigma\rangle\simeq \langle \sigma^3\rangle \times \langle \sigma^7\rangle $$ where $\sigma^3$ is a symplectic order 7 automorphism, and $\sigma^7=\operatorname{diag}(1,\omega,\omega^2,1,\omega,\omega^2)$ is the non-liftable (in the sense of \cite[Definition 4.4]{OY19}) non-symplectic automorphism of order 3. Similarly we have $$ \langle \nu\rangle\simeq \langle \nu^2\rangle\times\langle\nu^3\rangle $$ where $\nu^2$ is a symplectic automorphism of order 3, and $\nu^3$ is a non-symplectic involution. Hence  $$ G_s=\langle \sigma^3\rangle\rtimes\langle\nu^2\rangle\simeq(\ZZ/7\ZZ)\rtimes \ZZ/3\ZZ $$ and  $$ \overline{G}=\langle \sigma^7\rangle\times \langle\nu^3\rangle\simeq \ZZ/6\ZZ. $$ 

It is worth mentioning that these facts fit into the Hodge theoretic classification of symplectic automorphism groups of cubic fourfolds by Laza and Zheng \cite[Theorem 1.2]{LZ22}. 
\end{remark}

\section{Automorphisms of Hodge structures}
\label{sec4}

In the following sections we will study the induced action of the automorphism group of smooth projective hypersurfaces on its corresponding polarized Hodge structure. We refer the reader to the standard reference \cite{Voi02} for a detailed account on Hodge Theory.

Let $H$ be an integral polarized Hodge structure of weight $k$. Taking into account the Hodge-Riemann relations, one can embed $\Aut(H)$ as a discrete subgroup of a compact real Lie group 
$$\Aut(H)\hookrightarrow\prod_{p=0}^\frac{k-1}{2}O(2h^{p,k-p})\times O(h^{\frac{k}{2},\frac{k}{2}}),$$
and hence the group $\Aut(H)$ is finite.

The classical Torelli Theorem \cite{And58, arbarello2013geometry} states that for any pair of genus $g$ curves $X$ and $X'$, every isomorphism between their polarized Jacobians $J(X)\simeq J(X')$ is induced by a unique isomorphism of curves $X\simeq X'$ when $X$ is hyperelliptic. For non-hyperelliptic curves, given any isomorphism between polarized Jacobians $\varphi:J(X)\rightarrow J(X')$ there exists a unique isomorphism of curves $f:X\rightarrow X'$ such that $J(f)=\pm\varphi$. Since the Jacobian of a curve $X$ is determined by its primitive Hodge structure $J(X)=(\h^{1,0}(X))^*/\h_1(X,\ZZ)$, morphisms between polarized Jacobians correspond to morphisms between polarized Hodge structures. This motivates the following related conjectures, which are commonly referred as the Torelli Principle for hypersurfaces: Let $X$ and $X'$ be two smooth degree $d$ hypersurfaces of $\PP^{n+1}$
\begin{itemize}
    \item[(A)] Global Torelli Principle: An isomorphism of polarized Hodge structures $$\h^n(X,\ZZ)_{\prim}\simeq \h^n(X',\ZZ)_{\prim}$$ 
    which can be lifted to an isometry $\h^n(X,\ZZ)\simeq \h^n(X',\ZZ)$, preserving the corresponding power of the polarization $\theta_X^\frac{n}{2}\mapsto \theta_{X'}^\frac{n}{2}$ in the case $n$ even, implies $X\simeq X'$.
    \item[(B)] Strong Torelli Principle: If $X$ has an involution inducing $-\operatorname{Id}$ on $\h^n(X,\ZZ)_{\prim}$, every isomorphism of polarized Hodge structures
    $$\h^n(X,\ZZ)_{\prim}\simeq \h^n(X',\ZZ)_{\prim}$$
    which is induced by an isometry $\h^n(X,\ZZ)\simeq \h^n(X',\ZZ)$, preserving polarizations $\theta_X^\frac{n}{2}\mapsto\theta_{X'}^\frac{n}{2}$ in the case $n$ even, is induced by a unique isomorphism $X\simeq X'$. If $X$ has no involution of this type, every such isomorphism
    $$\varphi:\h^n(X,\ZZ)_{\prim}\rightarrow \h^n(X',\ZZ)_{\prim}$$
    induces a unique isomorphism $f:X'\rightarrow X$ such that $f^*=\pm\varphi$.
\end{itemize}
A consequence of the Strong Torelli Principle (also called the ``precise form'' of the Torelli Theorem by Serre's Appendix to \cite{Ser01}) is the following, which we call Punctual Torelli Principle.
\begin{itemize}
    \item[(C)] Punctual Torelli Principle: Let $X$ be a smooth degree $d$ hypersurface of $\PP^{n+1}$. If $X$ has an involution inducing $-\operatorname{Id}$ on $\h^n(X,\ZZ)_{\prim}$ then $$\Aut(X)\simeq \Aut(\h^n(X,\ZZ)_{\prim}),$$
    where the latter is the group of automorphisms of polarized Hodge structure induced by isometries of $\h^n(X,\ZZ)$, which preserve $\theta^\frac{n}{2}$ in the case $n$ even. If $X$ has no such involutions then 
    $$\Aut(X)\simeq \Aut(\h^n(X,\ZZ)_{\prim})/\{\pm 1\}.$$
\end{itemize}
Since we focus exclusively on the Punctual Torelli Principle, we will devote this section to the study of  automorphisms of polarized Hodge structures. Later in \S \ref{sec5} we will apply these general results to the case of suitable Klein hypersurfaces.

The first step towards the study of automorphisms of polarized Hodge structures is the following extension of \cite[Proposition 13.2.5]{BL04}, originally stated for principally polarized abelian varieties.

\begin{proposition}
\label{propHS}
Let $H$ be a polarized Hodge structure and $\sigma\in \Aut(H)$ be an automorphism of order $m$. Let $H_0\subseteq H$ be a $\sigma$-invariant subspace of dimension $h_0$. Consider the subspace
$$
\Fix\langle\sigma\rangle:=\langle\{v\in H_0: \sigma^i(v)=v \text{ for some }i=1,\ldots,m-1\}\rangle
$$
of dimension $h_1$. Then
$$
\varphi(m)\mid h_0-h_1.
$$
\end{proposition}

\begin{proof}
Since $\sigma$ induces an automorphism of $H':=H_0/ \Fix \langle\sigma\rangle$ such that $\sigma_\C^i$ has no fixed points for all $i=1,\ldots,m-1$, then all eigenvalues of $\sigma|_{H'}$ are $m$-th primitive roots of unity and so the characteristic polynomial of the rational representation of $\sigma|_{H'}$ must be $\mu_m(t)^\frac{\dim H'}{\varphi(m)}$, where $\mu_m$ is the cyclotomic polynomial and $\dim H'=h_0-h_1$.
\end{proof}

\begin{remark}
    Note that if $H_0$ is a sub-Hodge structure of $H$, then $\Fix\langle \sigma\rangle$ also is. Moreover if the order of $\sigma$ is a prime number, then $\Fix\langle \sigma\rangle=\Fix \sigma$, and the latter is defined over $\Q$. 
\end{remark}

In general the spectral decomposition of $\sigma$ might have a complicated relation with the Hodge decomposition. We content ourselves by analyzing the following extremal case.

\begin{definition}
\label{defextrpolHS}
We say that a polarized Hodge structure $H$ of rank $h$ is \textit{extremal} if it admits an automorphism $\sigma\in \Aut(H)$ of prime order $p=h-h_1+1$ where 
$
h_1:=\dim \Fix \sigma.
$
\end{definition}

\begin{corollary}
If $H$ is an extremal polarized Hodge structure with respect to the automorphism $\sigma$ of prime order $p$ and $H':=H/ \Fix \sigma$, then the spectrum of $\sigma|_{H'}\in \Aut(H')$ is the set of all $p$-th roots of unity with multiplicity one. Moreover, the weight $k$ Hodge structure of $H'$ induces a natural partition of the set $\{\zeta_p,\zeta_p^2,\ldots,\zeta_p^{p-1}\}=\sqcup_{r+s=k}C^{r,s}$ in such a way that $\overline{C^{r,s}}=C^{s,r}$.
\end{corollary}

\begin{proof}
Applying \cref{propHS} we see that all eigenspaces of $\sigma|_{H'}$ are 1-di\-men\-sion\-al. Since it is an automorphism of Hodge structure, its spectral decomposition refines the Hodge decomposition. The desired partition is induced by this refinement, considering in each $C^{r,s}$ the eigenvalues of the eigenspaces contained in ${H'}^{r,s}$.    
\end{proof}

The main theorem of this section is a classification of the possible groups that can appear as automorphisms of extremal polarized Hodge structures. This result is a direct adaptation of \cite[Th\'eor\`eme 2]{bennama1997remarques}. By the sake of completeness we reproduce the proof here. Let us recall first the following theorem of Brauer \cite[Theorem 4]{brauer1942groups}.

\begin{theorem}[Brauer]\label{brauer}
Let $G$ be a group of order $g=p\cdot g'$, where $p$ is a prime number and $\gcd(p,g')=1$. If $G$ has more than one subgroup of order $p$ and $G$ admits a faithful representation of degree $\frac{p-1}{2}$, then the quotient of $G$ by its center is isomorphic to $\operatorname{PSL}_2(\mathbf{F}_p)$.    
\end{theorem}

Using the above result we prove the main theorem of this section.

\begin{theorem}
\label{theoautextHS}
Let $H$ be an extremal polarized Hodge structure of odd weight $k$. Let $\sigma$ be its corresponding automorphism of prime order $p$ and $G:=\Aut(H)/\{\pm 1\}$, then $G\simeq \langle \sigma\rangle \rtimes \mu$ where $\mu<\Delta:=\{\tau\in \Gal(\Q(\zeta_p)/\Q): \tau(C^{r,s})=C^{r,s} \text{ for all }r+s=k\}$, or $G=\operatorname{PSL}_2(\mathbf{F}_p)$ with $[\Gal(\Q(\zeta_p)/\Q):\Delta]=2$ and $p\equiv -1\textup{ (mod 4)}$. 
\end{theorem}

\begin{proof}
The action of $\sigma$ on $H$ induces a structure of $\ZZ[\zeta_p]$-module on it, which in turn induces a $\Q(\zeta_p)$-vector space structure on $H\otimes_{\mathbf{Z}}\Q$. Since $H$ has rank $p-1$, once we fix a non-zero element $\beta_0\in H\otimes_{\mathbf{Z}}\Q$ we have an isomorphism of $\Q$-vector spaces
$$
\Q(\zeta_p)\simeq H\otimes_{\mathbf{Z}}\Q
$$
taking $1\mapsto \beta_0$. Under this isomorphism every $\tau_j\in\Gal(\Q(\zeta_p)/\Q)$ given by $\tau_j(\zeta_p)=\zeta_p^j$ induces an element $\ttau_j\in\GL(H\otimes_{\mathbf{Z}}\Q)$ given by $\ttau_j(\sigma^i(\beta_0))=\sigma^{i\cdot j}(\beta_0).$ From this we see that $\ttau_j\sigma\ttau_j^{-1}=\sigma^j$. On the other hand, let $\varphi\in N(\langle \sigma\rangle)$ be such that
$$
\varphi\sigma\varphi^{-1}=\sigma^j.
$$
It follows that $\sigma$ and $\sigma^j$ have the same eigenvalues when restricted to each ${H}^{r,s}$ and so we get a group homomorphism
$$
\varphi\in N_{\Aut(H)}(\langle \sigma\rangle)\mapsto \tau_j\in\Delta
$$
whose kernel is $C_{\Aut(H)}(\langle\sigma\rangle)$ and its image is $\mu<\Delta$. Given any $\varphi\in C_{\Aut(H)}(\langle\sigma\rangle)$ we have that $\varphi|_{H\otimes_{\mathbf{Z}}\Q}$ is a $\Q(\zeta_p)$-linear map, thus there exists some $u\in\Q(\sigma)^\times$ such that 
$$
\varphi(x)=u(x)
$$
for every $x\in H\otimes_{\mathbf{Z}}\Q$. Since $\varphi$ has finite order, it follows that $u$ is a root of unity and so it is of the form $u=\pm\sigma^i$. Thus $C_{\Aut(H)}(\langle \sigma\rangle)=\langle\sigma\rangle\times \{\pm 1\}$ and so $C_{G}(\langle \sigma\rangle)=\langle\sigma\rangle$. Since the order of $\mu$ and $\langle \sigma\rangle$ are coprime, it follows from Schur-Zassenhaus theorem \cite[Theorem 8.10]{Suz82} that
$$
N_G(\langle\sigma\rangle)\simeq \langle \sigma\rangle\rtimes \mu.
$$
The rest of the proof is purely group theoretical. We claim $\langle \sigma\rangle $ is a $p$-Sylow subgroup of $G$. In fact, consider $U$ a $p$-Sylow containing $\langle \sigma\rangle$. The center $Z(U)$ is a non-trivial $p$-group contained in $N_G(\langle\sigma\rangle)$ and $|N_G(\langle\sigma\rangle)|=p\cdot r$ where $\gcd(p,r)=1$ since $r=|\mu|\mid p-1$. It follows that $Z(U)=\langle\sigma\rangle$. Then $U\subseteq N_G(\langle \sigma\rangle)$ is a $p$-group and so $U=\langle \sigma\rangle$ as claimed. By Sylow theorem the number of $p$-Sylow subgroups of $G$ is $q:=[G:N_G(\langle \sigma\rangle)]\equiv 1$ (mod $p$). If $q=1$ then $G= N_G(\langle\sigma\rangle)\simeq\langle\sigma\rangle\rtimes\mu$ as claimed. If $q\neq 1$ we apply Brauer's \cref{brauer} and get that $G\simeq\operatorname{PSL}_2(\mathbf{F}_p)$. For this note that the complex representation of $G$ of degree $\frac{p-1}{2}$ is given by its natural restriction to the Hodge filtration $F^\frac{k+1}{2}H_\C$. The last assertion of the theorem follows from the fact that $N_G(\langle\sigma\rangle)$ is isomorphic to a Borel subgroup of $\operatorname{PSL}_2(\mathbf{F}_p)$ and so it has order $\frac{p(p-1)}{2}$. This implies $|\Delta|=\frac{p-1}{2}$ and so $\Delta\simeq (\mathbf{F}_p^\times)^2$. Finally to see that $p\equiv -1$ (mod $4$), note that $-1\notin (\mathbf{F}_p^\times)^2$ since the conjugation automorphism never belongs to $\Delta$ (because $\overline{C^{r,s}}=C^{s,r}$).
\end{proof}

\begin{remark}\label{rmk:autHS-ZvsQ}
In principle $\mu$ might be smaller than $\Delta$, however, every element of $\Delta$ represents an automorphism of the rational polarized Hodge structure $H\otimes_{\mathbf{Z}}\Q$. In fact, after the previous proof it is enough to show that every $\ttau_j\in\GL(H\otimes_{\mathbf{Z}}\Q)$ is an automorphism of rational polarized Hodge structure for every $\tau_j\in\Delta$. In order to see this we construct an explicit non-zero element $\beta_0\in H\otimes_{\mathbf{Z}}\Q$ in terms of the spectral decomposition of $H\otimes_{\mathbf{Z}}\C$ as follows: For every $i\in\{1,2,\ldots,p-1\}$ let $V(\zeta_p^i)\subseteq H\otimes_{\mathbf{Z}}\C$ be the 1-dimensional eigenspace corresponding to the eigenvalue $\zeta_p^i$. Hence
$
H\otimes_{\mathbf{Z}}\C=\bigoplus_{i=1}^{p-1}V(\zeta_p^i).
$
Each $V(\zeta_p^i)$ is defined over $\Q(\zeta_p)$, i.e., we can write
$
H\otimes_{\mathbf{Z}}\Q(\zeta_p)=\bigoplus_{i=1}^{p-1}V'(\zeta_p^i)
$
where $V'(\zeta_p^i)\otimes_{\mathbf{Z}}\C=V(\zeta_p^i)$ and each $V'(\zeta_p^i)$ is a 1-dimensional $\Q(\zeta_p)$-vector space. Let $\omega_1\in V'(\zeta_p)$ be a non-zero element. For every $j\in\{1,2,\ldots,p-1\}$ let $t_j\in \GL(H\otimes_{\mathbf{Z}}\Q(\zeta_p),\Q)$ be the $\Q$-linear map given by the Galois action of $\Gal(\Q(\zeta_p)/\Q)$ on $H\otimes_{\mathbf{Z}}\Q(\zeta_p)$, i.e. such that 
$
t_j(\zeta_p^i\cdot \beta)=\zeta_p^{i\cdot j}\cdot\beta \ \ \text{ for all }\beta\in H.
$
Clearly $t_j$ and $\sigma_{\Q(\zeta_p)}:=\sigma\otimes \operatorname{id}_{\Q(\zeta_p)}$ commute. Define 
$
\omega_j:=t_j(\omega_1)\in V'(\zeta_p^j).
$
It is easy to see that a basis of $H\otimes_{\mathbf{Z}} \Q$ is given by
$$
\beta_i:=\sum_{\ell=1}^{p-1}\zeta_p^{i\cdot \ell}\omega_\ell
$$
since clearly $t_j(\beta_i)=\sum_{\ell=1}^{p-1}\zeta_p^{i\cdot j\ell }\omega_{j\ell}=\beta_i$ and they are $\Q$-linearly independent. As in the proof above, let us assume the isomorphism $H\otimes_{\mathbf{Z}}\Q\simeq \Q(\zeta_p)$ is given by sending 
$
1\in\Q(\zeta_p)\mapsto\beta_0:=\sum_{\ell=1}^{p-1}\omega_\ell\in H\otimes_{\mathbf{Z}}\Q.
$
Thus it sends each $\zeta_p^i\mapsto \sigma^i(\beta_0)=\beta_i$, and so each $\ttau_j$ is given by
\begin{equation}
\label{eq1}
\ttau_j(\beta_i)=\beta_{i\cdot j}.
\end{equation}
Let $\langle\cdot,\cdot\rangle$ be the Hermitian product on $H\otimes_{\mathbf{Z}}\C$ induced by the polarization on $H\otimes_{\mathbf{Z}}\Q$. Since 
$
\langle \omega_i,\omega_j\rangle= \langle \sigma(\omega_i),\sigma(\omega_j)\rangle=\zeta_p^{i-j}\langle\omega_i,\omega_j\rangle
$
it follows that $\langle\omega_i,\omega_j\rangle=0$ if $i\neq j$. On the other hand $\langle\omega_i,\omega_i\rangle\in \Q(\zeta_p)\cap \R=\Q$. Then
$$
\langle\ttau_j(\beta_i),\ttau_j(\beta_h)\rangle=\langle\beta_{i\cdot j},\beta_{h\cdot j}\rangle=\sum_{\ell=1}^{p-1}\zeta_p^{j\cdot(i-h)\cdot\ell}\langle\omega_\ell,\omega_\ell\rangle=\tau_j(\langle\beta_i,\beta_h\rangle)=\langle\beta_i,\beta_h\rangle
$$
and so each $\ttau_j$ preserves the polarization. Finally, it follows from \eqref{eq1} that
$
(\ttau_j\otimes{\operatorname{id}_\C})(\omega_\ell)=\omega_{j^{-1}\cdot \ell}
$
and since $t_{j^{-1}}\in\Delta$ we get that $\ttau_j$ is in fact a morphism of rational Hodge structures.
\end{remark}

\section{Hypersurfaces of Wagstaff type}
\label{sec5}

In this section we will study the Rational Punctual Torelli Principle for some simple hypersurfaces which we call of Wagstaff type.

\begin{definition}
We say a smooth hypersurface $X\subseteq\mathbf{P}^{n+1}$ of degree $d$ is \textit{of Wagstaff type} if it admits an automorphism of prime order $p>(d-1)^n$.
\end{definition}

The above definition is based on the following result about primes realized as the order of an automorphism of a smooth hypersurface.

\begin{theorem}[{\cite{gl13}}]
\label{gal13}
Let $X\subseteq\mathbf{P}^{n+1}$ be a smooth hypersurface of dimension $n$ and degree $d$ admitting an automorphism of prime order $p$. Then $p<(d-1)^{n+1}$. Furthermore if $p>(d-1)^n$ then $X$ is isomorphic to the Klein hypersurface, $n=2$ and $p=(d-1)^2+1$, or $n+2$ is prime and $p=\frac{(d-1)^{n+2}+1}{d}$.  
\end{theorem}

\begin{remark}
The above result characterizes all hypersurfaces of Wagstaff type of dimension $n>2$ and degree $d$ as Klein hypersurfaces such that $n+2$ is prime and $\frac{(d-1)^{n+2}+1}{d}$ is also prime. Primes of the form 
$$
Q(b,q)=\frac{b^q+1}{b+1}
$$
are known as \textit{generalized Wagstaff primes base $b$} and are conjectured to be infinitely many. Classical Wagstaff primes are those of the form $Q(2,q)$, and Mersenne primes correspond to those of the form $Q(-2,q)$. A table with some generalized Wagstaff primes can be found in \cite{dubner2000primes}.
\end{remark}

\begin{proposition}
For every Klein hypersurface of Wagstaff type of dimension $n>2$ the middle primitive cohomology group is an extremal polarized Hodge structure. 
\end{proposition}

\begin{proof}
Let $X$ be a Klein hypersurface of Wagstaff type of dimension $n$ and degree $d$. Then by \cref{gal13} and \cref{remarkautdiag} there exists $\sigma\in \Aut(X)$ acting diagonally on $X=\{x_0^{d-1}x_1+x_1^{d-1}x_2+\cdots+x_n^{d-1}x_{n+1}+x_{n+1}^{d-1}x_0=0\}$ of prime order $p=\frac{(d-1)^{n+2}+1}{d}$ and given by
\begin{equation}
\label{autdiag}
\sigma(x_0:x_1:\cdots:x_{n+1})=(\zeta_px_0:\zeta_p^{1-d}x_1:\zeta_p^{(1-d)^2}x_2:\cdots:\zeta_p^{(1-d)^{n+1}}x_{n+1}).  
\end{equation}
The extremality of $H:=\h^n(X,\ZZ)_{\prim}$ follows from \cref{propHS} and the following well known equality (see for instance \cite[Corollary 1.1.12]{Huy23}) 
$$
\rank(H)=\dim_\C \h^n(X,\C)_{\prim}=\frac{(d-1)^{n+2}+1}{d}-1. 
$$
\end{proof}

\begin{theorem}
\label{thmtorellimodular}
Let $X=\{K=0\}$ be a Klein hypersurface of Wagstaff type of dimension $n\ge 3$ and degree $d\ge 3$, with $(n,d)\neq (3,3)$. Let $p=\frac{(d-1)^{n+2}+1}{d}$. For each $q\in\{0,1,\ldots,n\}$ consider the set
$$
S_q:=\left\{\sum_{i=0}^{n+1}\beta_i(1-d)^i\in \ZZ/p\ZZ: 0\le \beta_i\le d-2, \sum_{i=0}^{n+1}\beta_i=d(q+1)-n-2\right\}.
$$
Then $X$ satisfies the Punctual Torelli Principle as long as
\begin{equation}
\label{conditionTorelli}
\{m\in (\ZZ/p\ZZ)^\times: m\cdot S_q=S_q \ \ , \ \forall q=0,\ldots,n\}=\langle 1-d\rangle<(\ZZ/p\ZZ)^\times.
\end{equation}
\end{theorem}

\begin{proof}
It follows from Griffiths basis theorem \cite{gr69} that all classes of the form
$$
\omega_\beta=\res\left(\frac{x^\beta\Omega}{K^{q+1}}\right)^{n-q,q} \ \ , \hspace{1cm}\beta=(\beta_0,\ldots,\beta_{n+1})\in\{0,1,\ldots,d-2\}^{n+2}
$$
form a basis of eigenvectors for the automorphism of Hodge structure induced by $\sigma$ (the diagonal automorphism \eqref{autdiag} of order $p$). Furthermore
$$
\omega_\beta\in V\left(\zeta_p^{\sum_{i=0}^{n+1}\beta_i(1-d)^i}\right).
$$
Since $(1-d)$ has order $n+2$ modulo $p$, the result follows from \cref{theoautklein1}, \cref{theoautklein2} and \cref{theoautextHS}. Indeed, $\Aut(X)\simeq \ZZ/m\ZZ\rtimes \ZZ/(n+2)\ZZ $ acts faithfully on the middle cohomology by \cite[Corollary 1.3.18]{Huy23} and the previous results imply that $\Aut(\operatorname{H}^n(X,\ZZ)_{\prim}))/\{\pm 1\}<\langle \sigma\rangle\rtimes \Delta\simeq \ZZ/m\ZZ\rtimes\ZZ/(n+2)\ZZ$.
\end{proof}

\begin{remark}\label{rmk:Sq}
Note that each term of $S_q$ is in correspondence with some eigenvalue of the automorphism of the Hodge structure induced by the order $p$ diagonal automorphism $\sigma$ given by \eqref{autdiag}. The extremality of the Hodge structure implies that each eigenvalue has multiplicity one, i.e., $S_q$ cannot have repeated terms.
\end{remark}

\begin{example}\label{example:cubic-5fold}
Condition \eqref{conditionTorelli} is satisfied for the Klein cubic fivefold. In fact let $n=5$, $d=3$, $q=2$, then
\begin{align*}
S_2=&\left\{-21,-18,-16,-11,-10,-7,-6,-4,-1,\right. \\ &\left.2,3,5,8,9,12,13,14,15,17,19,20\right\}.    
\end{align*}

Checking element by element we get
$$
\{m\in (\ZZ/43\ZZ)^\times: m\cdot S_2=S_2 \}=\{-8,-2,1,4,11,16,21\}=\langle -2\rangle.
$$
Hence the Klein cubic fivefold satisfies the Punctual Torelli Principle. Note that in this case the Hodge structure is of level one, and so $\Aut(H)=\Aut(J(X))$ where $J(X)=(\h^{3,2}(X))^*/\h_5(X,\ZZ)\in\mathcal{A}_{21}$ is its polarized intermediate Jacobian. In particular, we deduce the isomorphism
\[
\operatorname{Aut}(X) \simeq \operatorname{Aut}(J(X)) \slash \{\pm 1\}.
\]
Similarly, if $X$ is the Klein quartic threefold in the 4-dimensional projective space and $J(X)=(\h^{2,1}(X))^*/\h_3(X,\ZZ)\in\mathcal{A}_{30}$ is the corresponding polarized intermediate Jacobian, then using the same argument we deduce that
\[
\operatorname{Aut}(X) \simeq \operatorname{Aut}(J(X)) \slash \{\pm 1\}.
\]
Our definition of extremal Hodge structure was inspired from these examples and the notion of extremal abelian variety used in the study of the zero dimensional stratum of the singular locus of $\mathcal{A}_g$ in \cite{GAMPZ05}. 
\end{example}

\begin{theorem}\label{thm:punctual-torelli}
The Punctual Torelli Principle holds for a Klein hypersurface of Wagstaff type of dimension $n\ge 3$ and degree $d\ge 3$ in the following cases:
\begin{itemize}
    \item[(a)] $d\mid n+3$.
    \item[(b)] $d=3$ and $n\ge 5$.
\end{itemize}
\end{theorem}

\begin{proof}
In both cases we have to verify that \eqref{conditionTorelli} holds. Following the same notation as in \cref{thmtorellimodular}, if $d\mid n+3$ there exists some $q\in\{0,1,\ldots,n\}$ such that $d(q+1)-n-2=1$, and so
$$
S_q=\langle 1-d\rangle<(\ZZ/p\ZZ)^\times.
$$
Then it is clear that $m\in (\ZZ/p\ZZ)^\times$ satisfies $m\cdot \langle 1-d\rangle=\langle 1-d\rangle$ if and only if $m\in\langle 1-d\rangle$. This shows (a). For the cubic case, since $n+2\ge 7$ is a prime number we have that $\gcd(3,n+2)=1$ and so $3\mid n+3$ or $3\mid n+4$. The case $3\mid n+3$ was already covered by (a). Suppose now that $3\mid n+4$, then there exists some $q\in\{0,1,\ldots,n\}$ such that $3(q+1)-n-2=2$ and so
$$
S_q=\{(-2)^i+(-2)^j\in\ZZ/p\ZZ: 0\le i<j\le n+1\}.
$$
We claim that if $m\cdot S_q=S_q$ then $m\in\langle -2\rangle<(\ZZ/p\ZZ)^\times$. This follows from an elementary combinatorial argument which is detailed in the Appendix. 
\end{proof}

We expect that \cref{thm:punctual-torelli} holds for all Klein hypersurfaces of Wagstaff type. This leads us to the following elementary number theory conjecture.    

\begin{conjecture}
\label{conj}
Consider natural numbers $n,d\ge 3$ with $(n,d)\neq (3,3)$, such that $n+2$ and $p:=\frac{(d-1)^{n+2}+1}{d}$ are prime numbers. For each $q\in\{0,1,\ldots,n\}$ let
$$
S_q:=\left\{\sum_{i=0}^{n+1}\beta_i(1-d)^i\in \ZZ/p\ZZ: 0\le \beta_i\le d-2, \sum_{i=0}^{n+1}\beta_i=d(q+1)-n-2\right\}.
$$
Then
$$
\{m\in (\ZZ/p\ZZ)^\times: m\cdot S_q=S_q \ \ , \ \forall q=0,\ldots,n\}=\langle 1-d\rangle<(\ZZ/p\ZZ)^\times.
$$
\end{conjecture}

We verified \cref{conj} computationally with the aid of a \texttt{Python} script kindly written for us by Fabián Levicán. In \cref{table}, we show the Klein hypersurfaces of Wagstaff type with $n\leq 15$ and $d\leq 11$. If the hypersurface is of Wagstaff type, the table gives the value of $p$. Otherwise, the table returns ``--''.

\begin{table}[!ht]
{\small  \begin{tabular}{ | c || c | c | c | c | c | c | c | c | c | c | c | c |}
    \hline	
    $n\backslash d$ & 3 & 4 & 5 & 6 & 7 & 8 & 9 & 10 & 11 \\
    \hline\hline
    3 & -- & 61 & -- & 521 & -- & -- & -- & -- & 9091 \\
    \hline
    5 & 43 & 547 & -- & -- & -- & -- & -- & -- & 909091 \\
    \hline
    9 & 683 & -- & -- & -- & 51828151 & -- & -- & -- & -- \\ \hline
    11 & 2731 & 398581 & -- & -- & -- & -- & -- & -- & -- \\    \hline
    15 & 43691 & -- & -- & -- & -- & 29078814248401 & -- 
    & -- & -- \\    \hline
  \end{tabular}} \vspace{2ex}
  \caption{Klein hypersurfaces of Wagstaff type.}
  \label{table}
\end{table}

\begin{remark}
\label{rmkk3}
For the Klein quartic surface $X\subseteq \mathbf{P}^3$, the Strong Torelli Theorem holds \cite{pyatetskii1971torelli}, but it is the only Klein hypersurface with unknown automorphism group. Its Picard rank was computed by Shioda in \cite[\S 4]{Shi86}, but we can give an alternative proof using automorphisms. Indeed, despite it is not of Wagstaff type, we can still use the explicit action of the diagonal automorphism 
$$
\sigma=\operatorname{diag}(\zeta_{80},\zeta_{80}^{-3},\zeta_{80}^9,\zeta_{80}^{-27})
$$
to compute the spectral decomposition of $\h^2(X,\C)_{\prim}$ obtaining that
$$
\h^{2,0}(X)\oplus \h^{0,2}(X)=V(-i)\oplus V(i).
$$
In consequence, $\h^{2,0}(X)\oplus \h^{0,2}(X)$ is invariant under the Galois action of the group $\Gal(\Q(\zeta_{20})/\Q)$, and so it is defined over $\Q$. By \cite[Proposition 1]{Bea14} this implies that $\rho(X)=h^{1,1}(X)=20$. As a consequence of the previous computations, together with a classical result by Shioda and Inose \cite[Theorem 5]{SI77}, we have that the group $\operatorname{Aut}(X)$ is infinite.
\end{remark}

\section*{Appendix: Proof of \cref{conj} for $d=3$} \label{appendix}

In this appendix we provide an elementary combinatorial argument to show that \eqref{conditionTorelli} always holds for $d=3$, completing in this way the proof of \cref{thm:punctual-torelli} (b).

To do so, we may assume that $n\geq 11$ since the case $n=5$ is covered by \cref{example:cubic-5fold} and the case $n=9$ was already treated in \cref{thm:punctual-torelli} (a). Consider $m\in \ZZ/p\ZZ$ such that $m\cdot S_q=S_q$. Then
\begin{equation}
\label{igualdad1}
m(1+(-2))=(-2)^i+(-2)^j.
\end{equation}
If $|j-i|\equiv1$ (mod $n+2$), then $(-2)^i+(-2)^j=(-2)^a(1+(-2))$ for some $a\in\{i,j\}$. In fact, if $j=i+1+(n+2)k$ then $(-2)^i + (-2)^j = (-2)^i(1+(-2))$ as $(-2)$ has order $n+2$ in $(\mathbf{Z}\slash p\mathbf{Z})^\times$, and similarly if $i-j\equiv 1$ (mod $n+2$). Therefore, $
m=(-2)^a
$
as claimed. 

If $|j-i|\not\equiv 1$ (mod $n+2$) consider the following sets:
$$
A:=\{a\in\{2,3,\ldots,n+1\}: m(1+(-2)^a)=(-2)^r+(-2)^s \ , \ \{r,s\}\cap\{i,j\}=\varnothing\},
$$
$$
B:=\{b\in\{2,3,\ldots,n+1\}: m(1+(-2)^b)=(-2)^i+(-2)^s \ , \ s\neq j\},
$$
$$
C:=\{c\in\{2,3,\ldots,n+1\}: m(1+(-2)^c)=(-2)^r+(-2)^j \ , \ r\neq i\}.
$$
They form a partition of the set
$$
\{2,3,\ldots,n+1\}=A\sqcup B\sqcup C.
$$
Let us separate the analysis into three cases:

\vspace{2mm}

\noindent {\textbf{Case 1}:} $\#A\ge 4$.

\vspace{2mm}

In this case pick some $a\in A$ such that 
\begin{equation}
\label{igualdad2}
m(1+(-2)^a)=(-2)^r+(-2)^s.
\end{equation}
If $|a|=|r-s|$, by a similar argument as above (where $|j-i|\equiv 1$ (mod $n+2$)) we can show that $m=(-2)^{a'}$ for some $a'$ as claimed. If not, consider the quadrilateral corresponding to the four points $i,j,r,s\in\ZZ/(n+2)\ZZ$ (see \cref{Fig1}). 

\begin{figure}[ht]
    \centering
\begin{tikzpicture}[x=0.75pt,y=0.75pt,yscale=-1,xscale=1,scale=0.45]

\draw  [dash pattern={on 0.84pt off 2.51pt}] (203,151.32) .. controls (203,80.69) and (260.26,23.42) .. (330.9,23.42) .. controls (401.54,23.42) and (458.8,80.69) .. (458.8,151.32) .. controls (458.8,221.96) and (401.54,279.22) .. (330.9,279.22) .. controls (260.26,279.22) and (203,221.96) .. (203,151.32) -- cycle ;
\draw   (330.9,23.42) -- (389.75,36.83) -- (432.75,73.83) -- (455.7,121.6) -- (456.7,174.45) -- (434.2,227.95) -- (397.2,260.95) -- (354.75,275.83) -- (303.75,275.83) -- (263.2,259.45) -- (224.75,222.83) -- (205.75,176.83) -- (206.75,118.42) -- (231.75,69.42) -- (276.75,35.42) -- cycle ;
\draw [color={rgb, 255:red, 0; green, 0; blue, 0 }  ,draw opacity=1 ]   (231.75,69.42) -- (432.75,73.83) ;
\draw [shift={(432.75,73.83)}, rotate = 1.25] [color={rgb, 255:red, 0; green, 0; blue, 0 }  ,draw opacity=1 ][fill={rgb, 255:red, 0; green, 0; blue, 0 }  ,fill opacity=1 ][line width=0.75]      (0, 0) circle [x radius= 3.35, y radius= 3.35]   ;
\draw [shift={(231.75,69.42)}, rotate = 1.25] [color={rgb, 255:red, 0; green, 0; blue, 0 }  ,draw opacity=1 ][fill={rgb, 255:red, 0; green, 0; blue, 0 }  ,fill opacity=1 ][line width=0.75]      (0, 0) circle [x radius= 3.35, y radius= 3.35]   ;
\draw [color={rgb, 255:red, 74; green, 144; blue, 226 }  ,draw opacity=1 ]   (432.75,73.83) -- (456.7,174.45) ;
\draw [color={rgb, 255:red, 208; green, 2; blue, 27 }  ,draw opacity=1 ]   (231.75,69.42) -- (224.75,222.83) ;
\draw [color={rgb, 255:red, 0; green, 0; blue, 0 }  ,draw opacity=1 ]   (456.7,174.45) -- (224.75,222.83) ;
\draw [shift={(224.75,222.83)}, rotate = 168.22] [color={rgb, 255:red, 0; green, 0; blue, 0 }  ,draw opacity=1 ][fill={rgb, 255:red, 0; green, 0; blue, 0 }  ,fill opacity=1 ][line width=0.75]      (0, 0) circle [x radius= 3.35, y radius= 3.35]   ;
\draw [shift={(456.7,174.45)}, rotate = 168.22] [color={rgb, 255:red, 0; green, 0; blue, 0 }  ,draw opacity=1 ][fill={rgb, 255:red, 0; green, 0; blue, 0 }  ,fill opacity=1 ][line width=0.75]      (0, 0) circle [x radius= 3.35, y radius= 3.35]   ;

\draw (210,49.4) node [anchor=north west][inner sep=0.75pt]  [color={rgb, 255:red, 0; green, 0; blue, 0 }  ,opacity=1 ]  {$i$};
\draw (200,225.4) node [anchor=north west][inner sep=0.75pt]  [color={rgb, 255:red, 0; green, 0; blue, 0 }  ,opacity=1 ]  {$j$};
\draw (466,173.4) node [anchor=north west][inner sep=0.75pt]  [color={rgb, 255:red, 0; green, 0; blue, 0 }  ,opacity=1 ]  {$r$};
\draw (445,56.4) node [anchor=north west][inner sep=0.75pt]  [color={rgb, 255:red, 0; green, 0; blue, 0 }  ,opacity=1 ]  {$s$};
\draw (232,135.4) node [anchor=north west][inner sep=0.75pt]  [color={rgb, 255:red, 208; green, 2; blue, 27 }  ,opacity=1 ]  {$|j-i|$};
\draw (346,116.4) node [anchor=north west][inner sep=0.75pt]  [color={rgb, 255:red, 74; green, 144; blue, 226 }  ,opacity=1 ]  {$|r-s|$};
\end{tikzpicture}
\caption{} \label{Fig1}
\end{figure}

If $|r-s|\equiv |j-i|$ (mod $n+2$), then 
$$
m(1+(-2)^a)=(-2)^r+(-2)^s
=(-2)^{\alpha}((-2)^i+(-2)^j)
=m((-2)^{\alpha}+(-2)^{\alpha+1}),
$$
for some $\alpha$ and so 
$$
1+(-2)^a=(-2)^{\alpha}+(-2)^{\alpha+1},
$$
which is impossible since $S_q$ has no repeated terms by \cref{rmk:Sq}. 

We claim that there exists $a'\in A\setminus \{a\}$ such that 
\begin{equation}
\label{igualdad3}
m(1+(-2)^{a'})=(-2)^{r'}+(-2)^{s'},
\end{equation}
and the quadrilateral $\{i,j,r,s\}$ is not a rotation of $\{i,j,r',s'\}$. Otherwise, the quadrilateral $\{i,j,r,s\}$ has another side or diagonal of length $|j-i|$, and it is elementary to see that there are at most two of them (see \cref{Fig2}).

\begin{figure}[ht]
    \centering
\begin{tikzpicture}[x=0.75pt,y=0.75pt,yscale=-1,xscale=1,scale=0.45]

\draw  [dash pattern={on 0.84pt off 2.51pt}] (38,148.32) .. controls (38,77.69) and (95.26,20.42) .. (165.9,20.42) .. controls (236.54,20.42) and (293.8,77.69) .. (293.8,148.32) .. controls (293.8,218.96) and (236.54,276.22) .. (165.9,276.22) .. controls (95.26,276.22) and (38,218.96) .. (38,148.32) -- cycle ;
\draw   (165.9,20.42) -- (224.75,33.83) -- (267.75,70.83) -- (290.7,118.6) -- (291.7,171.45) -- (269.2,224.95) -- (232.2,257.95) -- (189.75,272.83) -- (138.75,272.83) -- (98.2,256.45) -- (59.75,219.83) -- (40.75,173.83) -- (41.75,115.42) -- (66.75,66.42) -- (111.75,32.42) -- cycle ;
\draw  [dash pattern={on 0.84pt off 2.51pt}] (371,148.32) .. controls (371,77.69) and (428.26,20.42) .. (498.9,20.42) .. controls (569.54,20.42) and (626.8,77.69) .. (626.8,148.32) .. controls (626.8,218.96) and (569.54,276.22) .. (498.9,276.22) .. controls (428.26,276.22) and (371,218.96) .. (371,148.32) -- cycle ;
\draw   (498.9,20.42) -- (557.75,33.83) -- (600.75,70.83) -- (623.7,118.6) -- (624.7,171.45) -- (602.2,224.95) -- (565.2,257.95) -- (522.75,272.83) -- (471.75,272.83) -- (431.2,256.45) -- (392.75,219.83) -- (373.75,173.83) -- (374.75,115.42) -- (399.75,66.42) -- (444.75,32.42) -- cycle ;
\draw [color={rgb, 255:red, 208; green, 2; blue, 27 }  ,draw opacity=1 ]   (165.9,20.42) -- (59.75,219.83) ;
\draw    (269.2,224.95) -- (59.75,219.83) ;
\draw    (269.2,224.95) -- (165.9,20.42) ;
\draw [shift={(165.9,20.42)}, rotate = 243.2] [color={rgb, 255:red, 0; green, 0; blue, 0 }  ][fill={rgb, 255:red, 0; green, 0; blue, 0 }  ][line width=0.75]      (0, 0) circle [x radius= 3.35, y radius= 3.35]   ;
\draw [shift={(269.2,224.95)}, rotate = 243.2] [color={rgb, 255:red, 0; green, 0; blue, 0 }  ][fill={rgb, 255:red, 0; green, 0; blue, 0 }  ][line width=0.75]      (0, 0) circle [x radius= 3.35, y radius= 3.35]   ;
\draw    (165.9,20.42) -- (232.2,257.95) ;
\draw    (59.75,219.83) -- (232.2,257.95) ;
\draw [shift={(232.2,257.95)}, rotate = 12.47] [color={rgb, 255:red, 0; green, 0; blue, 0 }  ][fill={rgb, 255:red, 0; green, 0; blue, 0 }  ][line width=0.75]      (0, 0) circle [x radius= 3.35, y radius= 3.35]   ;
\draw [shift={(59.75,219.83)}, rotate = 12.47] [color={rgb, 255:red, 0; green, 0; blue, 0 }  ][fill={rgb, 255:red, 0; green, 0; blue, 0 }  ][line width=0.75]      (0, 0) circle [x radius= 3.35, y radius= 3.35]   ;
\draw [color={rgb, 255:red, 74; green, 144; blue, 226 }  ,draw opacity=1 ]   (269.2,224.95) -- (232.2,257.95) ;
\draw [color={rgb, 255:red, 208; green, 2; blue, 27 }  ,draw opacity=1 ]   (392.75,219.83) -- (498.9,20.42) ;
\draw    (392.75,219.83) -- (602.2,224.95) ;
\draw [shift={(602.2,224.95)}, rotate = 1.4] [color={rgb, 255:red, 0; green, 0; blue, 0 }  ][fill={rgb, 255:red, 0; green, 0; blue, 0 }  ][line width=0.75]      (0, 0) circle [x radius= 3.35, y radius= 3.35]   ;
\draw [shift={(392.75,219.83)}, rotate = 1.4] [color={rgb, 255:red, 0; green, 0; blue, 0 }  ][fill={rgb, 255:red, 0; green, 0; blue, 0 }  ][line width=0.75]      (0, 0) circle [x radius= 3.35, y radius= 3.35]   ;
\draw    (498.9,20.42) -- (602.2,224.95) ;
\draw [color={rgb, 255:red, 126; green, 211; blue, 33 }  ,draw opacity=1 ]   (373.75,173.83) -- (602.2,224.95) ;
\draw    (373.75,173.83) -- (498.9,20.42) ;
\draw [shift={(498.9,20.42)}, rotate = 309.21] [color={rgb, 255:red, 0; green, 0; blue, 0 }  ][fill={rgb, 255:red, 0; green, 0; blue, 0 }  ][line width=0.75]      (0, 0) circle [x radius= 3.35, y radius= 3.35]   ;
\draw [shift={(373.75,173.83)}, rotate = 309.21] [color={rgb, 255:red, 0; green, 0; blue, 0 }  ][fill={rgb, 255:red, 0; green, 0; blue, 0 }  ][line width=0.75]      (0, 0) circle [x radius= 3.35, y radius= 3.35]   ;
\draw    (392.75,219.83) -- (373.75,173.83) ;

\draw (252.7,244.85) node [anchor=north west][inner sep=0.75pt]  [color={rgb, 255:red, 74; green, 144; blue, 226 }  ,opacity=1 ]  {$|r-s|$};
\draw (149,-10) node [anchor=north west][inner sep=0.75pt]  [color={rgb, 255:red, 0; green, 0; blue, 0 }  ,opacity=1 ]  {$i$};
\draw (46,223.4) node [anchor=north west][inner sep=0.75pt]  [color={rgb, 255:red, 0; green, 0; blue, 0 }  ,opacity=1 ]  {$j$};
\draw (276,218.4) node [anchor=north west][inner sep=0.75pt]  [color={rgb, 255:red, 0; green, 0; blue, 0 }  ,opacity=1 ]  {$s$};
\draw (234.2,261.35) node [anchor=north west][inner sep=0.75pt]  [color={rgb, 255:red, 0; green, 0; blue, 0 }  ,opacity=1 ]  {$r$};
\draw (105,129.4) node [anchor=north west][inner sep=0.75pt]  [color={rgb, 255:red, 208; green, 2; blue, 27 }  ,opacity=1 ]  {$|j-i|$};
\draw (484,-10) node [anchor=north west][inner sep=0.75pt]  [color={rgb, 255:red, 0; green, 0; blue, 0 }  ,opacity=1 ]  {$i$};
\draw (370,222.4) node [anchor=north west][inner sep=0.75pt]  [color={rgb, 255:red, 0; green, 0; blue, 0 }  ,opacity=1 ]  {$j$};
\draw (343,165.4) node [anchor=north west][inner sep=0.75pt]  [color={rgb, 255:red, 0; green, 0; blue, 0 }  ,opacity=1 ]  {$r'$};
\draw (610,224.4) node [anchor=north west][inner sep=0.75pt]  [color={rgb, 255:red, 0; green, 0; blue, 0 }  ,opacity=1 ]  {$s'$};
\draw (452,100.4) node [anchor=north west][inner sep=0.75pt]  [color={rgb, 255:red, 208; green, 2; blue, 27 }  ,opacity=1 ]  {$|j-i|$};
\draw (436.7,150.4) node [anchor=north west][inner sep=0.75pt]  [color={rgb, 255:red, 126; green, 211; blue, 33 }  ,opacity=1 ]  {$|r'-s'|$};

\end{tikzpicture}
\caption{} \label{Fig2}

\end{figure}

Say that the sides and diagonals (different from $\{i,j\}$) of length $|j-i|$ are $e_1$ and $e_2$. Let $f_k:=\{i,j,r,s\}\setminus e_k$ be its complementary side or diagonal (see \cref{Fig3}).

\begin{figure}[ht]
    \centering
\begin{tikzpicture}[x=0.75pt,y=0.75pt,yscale=-1,xscale=1,scale=0.5]

\draw  [dash pattern={on 0.84pt off 2.51pt}] (205,159.38) .. controls (205,88.74) and (262.26,31.47) .. (332.9,31.47) .. controls (403.54,31.47) and (460.8,88.74) .. (460.8,159.38) .. controls (460.8,230.01) and (403.54,287.28) .. (332.9,287.28) .. controls (262.26,287.28) and (205,230.01) .. (205,159.38) -- cycle ;
\draw   (332.9,31.47) -- (391.75,44.88) -- (434.75,81.88) -- (457.7,129.65) -- (458.7,182.5) -- (436.2,236) -- (399.2,269) -- (356.75,283.88) -- (305.75,283.88) -- (265.2,267.5) -- (226.75,230.88) -- (207.75,184.88) -- (208.75,126.47) -- (233.75,77.47) -- (278.75,43.47) -- cycle ;
\draw [color={rgb, 255:red, 208; green, 2; blue, 27 }  ,draw opacity=1 ]   (332.9,31.47) -- (226.75,230.88) ;
\draw [color={rgb, 255:red, 208; green, 2; blue, 27 }  ,draw opacity=1 ]   (436.2,236) -- (226.75,230.88) ;
\draw [color={rgb, 255:red, 208; green, 2; blue, 27 }  ,draw opacity=1 ]   (436.2,236) -- (332.9,31.47) ;
\draw [shift={(332.9,31.47)}, rotate = 243.2] [color={rgb, 255:red, 208; green, 2; blue, 27 }  ,draw opacity=1 ][fill={rgb, 255:red, 208; green, 2; blue, 27 }  ,fill opacity=1 ][line width=0.75]      (0, 0) circle [x radius= 3.35, y radius= 3.35]   ;
\draw [shift={(436.2,236)}, rotate = 243.2] [color={rgb, 255:red, 208; green, 2; blue, 27 }  ,draw opacity=1 ][fill={rgb, 255:red, 208; green, 2; blue, 27 }  ,fill opacity=1 ][line width=0.75]      (0, 0) circle [x radius= 3.35, y radius= 3.35]   ;
\draw [color={rgb, 255:red, 126; green, 211; blue, 33 }  ,draw opacity=1 ]   (332.9,31.47) -- (399.2,269) ;
\draw [color={rgb, 255:red, 126; green, 211; blue, 33 }  ,draw opacity=1 ]   (226.75,230.88) -- (399.2,269) ;
\draw [shift={(399.2,269)}, rotate = 12.47] [color={rgb, 255:red, 126; green, 211; blue, 33 }  ,draw opacity=1 ][fill={rgb, 255:red, 126; green, 211; blue, 33 }  ,fill opacity=1 ][line width=0.75]      (0, 0) circle [x radius= 3.35, y radius= 3.35]   ;
\draw [shift={(226.75,230.88)}, rotate = 12.47] [color={rgb, 255:red, 126; green, 211; blue, 33 }  ,draw opacity=1 ][fill={rgb, 255:red, 126; green, 211; blue, 33 }  ,fill opacity=1 ][line width=0.75]      (0, 0) circle [x radius= 3.35, y radius= 3.35]   ;
\draw [color={rgb, 255:red, 74; green, 144; blue, 226 }  ,draw opacity=1 ]   (436.2,236) -- (399.2,269) ;

\draw (316,5.45) node [anchor=north west][inner sep=0.75pt]  [color={rgb, 255:red, 0; green, 0; blue, 0 }  ,opacity=1 ]  {$i$};
\draw (213,234.45) node [anchor=north west][inner sep=0.75pt]  [color={rgb, 255:red, 0; green, 0; blue, 0 }  ,opacity=1 ]  {$j$};
\draw (443,229.45) node [anchor=north west][inner sep=0.75pt]  [color={rgb, 255:red, 0; green, 0; blue, 0 }  ,opacity=1 ]  {$s$};
\draw (401.2,272.4) node [anchor=north west][inner sep=0.75pt]  [color={rgb, 255:red, 0; green, 0; blue, 0 }  ,opacity=1 ]  {$r$};
\draw (322,207.45) node [anchor=north west][inner sep=0.75pt]  [color={rgb, 255:red, 208; green, 2; blue, 27 }  ,opacity=1 ]  {$e_{1}$};
\draw (393,118.45) node [anchor=north west][inner sep=0.75pt]  [color={rgb, 255:red, 208; green, 2; blue, 27 }  ,opacity=1 ]  {$e_{2}$};
\draw (342,157.45) node [anchor=north west][inner sep=0.75pt]  [color={rgb, 255:red, 126; green, 211; blue, 33 }  ,opacity=1 ]  {$f_{1}$};
\draw (290,248.45) node [anchor=north west][inner sep=0.75pt]  [color={rgb, 255:red, 126; green, 211; blue, 33 }  ,opacity=1 ]  {$f_{2}$};

\end{tikzpicture}
\caption{} \label{Fig3}
\end{figure}

Since $\#(A\setminus \{a\})\ge 3$, we can pick some $a'\in A$ satisfying \eqref{igualdad3} such that $|r'-s'|$ is different to the lengths of $f_1$ and $f_2$. Thus, $\{i,j,r',s'\}$ cannot be a rotation of $\{i,j,r,s\}$.

Adding \eqref{igualdad1} and \eqref{igualdad2}, we obtain
$$
m(-2)^a=(-2)^i+(-2)^j+(-2)^r+(-2)^s
$$
and so
\begin{equation}
\label{igualdad4}    
m=(-2)^{i-a}+(-2)^{j-a}+(-2)^{r-a}+(-2)^{s-a}.
\end{equation}
Similarly, adding \eqref{igualdad1} and \eqref{igualdad3}, we get
\begin{equation}
\label{igualdad5}    
m=(-2)^{i-a'}+(-2)^{j-a'}+(-2)^{r'-a'}+(-2)^{s'-a'}.
\end{equation}
Since $n\geq 7$, we can consider some $\theta\in\ZZ/(n+2)\ZZ\setminus\{i-a,j-a,r-a,s-a,i-a',j-a',r'-a',s'-a'\}$. The following number
\begin{align*}
(-2)^{i-a}+(-2)^{j-a}+&(-2)^{r-a}+(-2)^{s-a}+(-2)^\theta= \\&(-2)^{i-a'}+(-2)^{j-a'}+(-2)^{r'-a'}+(-2)^{s'-a'}+(-2)^\theta    
\end{align*}
has two different ways to be written as sum of five different powers of $(-2)$ in modulo $p$. In other words, the set $S_{q+1}$ has some repeated terms, contradicting \cref{rmk:Sq}. This contradiction completes the proof when $\#A\ge 4$.

\vspace{2mm}

\noindent \textbf{Case 2}: $\#A\le 3$ and $|j-i|\not\equiv 2$ (mod $n+2$).

\vspace{2mm}

In this case $\#B+\#C\ge n-3\ge 8$ thus one of them, say $B$, has at least $\#B\ge 4$. Pick some $b,b'\in B$ such that
\begin{equation}
\label{igualdad6}
m(1+(-2)^b)=(-2)^i+(-2)^s,   
\end{equation}
\begin{equation}
\label{igualdad7}
m(1+(-2)^{b'})=(-2)^i+(-2)^{s'},   
\end{equation}
with $s,s'\in\ZZ/(n+2)\ZZ\setminus\{i+1,i+2,j\}$. Adding \eqref{igualdad1} and \eqref{igualdad6} we get
$$
m(-2)^b=2(-2)^i+(-2)^j+(-2)^s=(-2)^{i+1}+(-2)^{i+2}+(-2)^j+(-2)^s
$$
and so
\begin{equation}
\label{igualdad8}
m=(-2)^{i+1-b}+(-2)^{i+2-b}+(-2)^{j-b}+(-2)^{s-b}.
\end{equation}
Similarly, adding \eqref{igualdad1} and \eqref{igualdad7} we get
\begin{equation}
\label{igualdad9}
m=(-2)^{i+1-b'}+(-2)^{i+2-b'}+(-2)^{j-b'}+(-2)^{s'-b'}.
\end{equation}
Since \eqref{igualdad8} and \eqref{igualdad9} are two different ways of writing $m$, adding some $(-2)^\theta$ for some $\theta\in\ZZ/(n+2)\ZZ\setminus\{i+1-b,i+2-b,j-b,s-b,i+1-b',i+2-b',j-b',s'-b'\}$ we get a contradiction as in Case 1.

\vspace{2mm}

\noindent{\textbf{Case 3}:} $\#A\le 3$ and $|j-i|\equiv 2$ (mod $n+2$).

\vspace{2mm}

In this case we can take $j=i+2$. Again, $\#B+\#C\ge n-3\ge 8$. If $\#C\ge 4$, then we can pick $c,c'\in C$ such that
\begin{equation}
\label{igualdad10}
m(1+(-2)^c)=(-2)^r+(-2)^j,   
\end{equation}
\begin{equation}
\label{igualdad11}
m(1+(-2)^{c'})=(-2)^{r'}+(-2)^j,   
\end{equation}
with $r,r'\in\ZZ/(n+2)\ZZ\setminus\{j+1,j+2,i\}$, and the contradiction follows in the same way as in Case 2. If $\#C\le 3$, then $\#B\ge 5$ and we can pick $b,b'\in B$ such that
\begin{equation}
\label{igualdad12}
m(1+(-2)^b)=(-2)^i+(-2)^s,   
\end{equation}
\begin{equation}
\label{igualdad13}
m(1+(-2)^{b'})=(-2)^i+(-2)^{s'},   
\end{equation}
with $s,s'\in\ZZ/(n+2)\ZZ\setminus\{i+1,i+3,i+4\}$. Adding \eqref{igualdad1} and \eqref{igualdad12} we get
$$
m(-2)^b=2(-2)^i+(-2)^{i+2}+(-2)^s=(-2)^{i+1}+(-2)^{i+3}+(-2)^{i+4}+(-2)^s
$$
and so
\begin{equation}
\label{igualdad14}
m=(-2)^{i+1-b}+(-2)^{i+3-b}+(-2)^{i+4-b}+(-2)^{s-b}.
\end{equation}
Similarly with \eqref{igualdad1} and \eqref{igualdad13} we get
\begin{equation}
\label{igualdad15}
m=(-2)^{i+1-b'}+(-2)^{i+3-b'}+(-2)^{i+4-b'}+(-2)^{s'-b'}.
\end{equation}
From this we get a contradiction in the same way as in Case 1.

\bibliographystyle{alpha}
\bibliography{nfolds}
\end{document}